\documentclass[12pt]{amsart}

\usepackage{amsthm,amssymb,amsmath,amsfonts, tikz, graphicx, a4wide,subfigure,bm}
\usepackage[english]{babel}
\usepackage{subfigure}
\usepackage{mathtools}
\usepackage{url}
\usepackage{enumerate}
\usepackage[hidelinks]{hyperref}

\DeclarePairedDelimiter\floor{\lfloor}{\rfloor}
\DeclarePairedDelimiter{\ceil}{\lceil}{\rceil}

\usepackage{xargs}   
\usepackage{xcolor}
\usepackage[colorinlistoftodos,prependcaption]{todonotes}
\newcommandx{\unsure}[2][1=]{\todo[linecolor=red,backgroundcolor=red!25,bordercolor=red,#1]{#2}}
\newcommandx{\change}[2][1=]{\todo[linecolor=blue,backgroundcolor=blue!25,bordercolor=blue,#1]{#2}}
\newcommandx{\info}[2][1=]{\todo[linecolor=OliveGreen,backgroundcolor=OliveGreen!25,bordercolor=OliveGreen,#1]{#2}}
\newcommandx{\improvement}[2][1=]{\todo[linecolor=Plum,backgroundcolor=red!25,bordercolor=red,#1]{#2}}
\newcommandx{\thiswillnotshow}[2][1=]{\todo[disable,#1]{#2}}

\begin{document}

\newtheorem{prop}{Proposition}[section]
\newtheorem{theorem}{Theorem}[section]
\newtheorem{lemma}{Lemma}[section]
\newtheorem{cor}{Corollary}[section]
\newtheorem{remark}{Remark}[section]
\theoremstyle{definition}
\newtheorem{defn}{Definition}[section]
\newtheorem{ex}{Example}[section]

\numberwithin{equation}{section}

\title{Decay of correlations for critically intermittent systems}
\author[Kalle, Zeegers]{Charlene Kalle and Benthen Zeegers}

%
\address{C.C.C.J. Kalle\\ Mathematical Institute, University of Leiden, PO Box 9512, 2300 RA Leiden, The Netherlands}
\email{kallecccj@math.leidenuniv.nl}
%
%
%
%
\address{B.P. Zeegers\\ Mathematical Institute, University of Leiden, PO Box 9512, 2300 RA Leiden, The Netherlands}
\email{b.p.zeegers@math.leidenuniv.nl}

\date{Version of \today}

\begin{abstract}
For a family of random intermittent dynamical systems with a superattracting fixed point we prove that a phase transition occurs between the existence of an absolutely continuous invariant probability measure and infinite measure depending on the randomness parameters and the orders of the maps at the superattracting fixed point. In case the systems have an absolutely continuous invariant probability measure, we show that the systems are mixing and that the correlations decay polynomially even though some of the deterministic maps present in the system have exponential decay. This contrasts other known results, where random systems adopt the best decay rate of the deterministic maps in the systems.
\end{abstract}

\maketitle

\section{Introduction}\label{sec1}

Intermittency is a type of behaviour observed in certain dynamical systems, where the system alternates between long periods of either irregular activity or being in a seemingly steady state. Manneville and Pomeau investigated several different types of intermittency in the context of transitions to turbulence in convective fluids, see \cite{ManPum,ManPum2,BPV}. Well-known examples of one-dimensional systems exhibiting intermittent behaviour are the Manneville-Pomeau maps given by
\begin{equation}\label{q:mpmaps}
T_{\alpha}: [0,1]\to [0,1], \, x \mapsto x+x^{1+\alpha} \pmod 1, \quad \alpha >0,
\end{equation}
and their adaptations introduced in \cite{LSV} by Liverani, Saussol and Vaienti, now called LSV maps, that are given by
\begin{equation}\label{q:lsvmaps}
S_{\alpha}: [0,1] \to [0,1], \, x \mapsto
\begin{cases}
x(1+2^{\alpha} x^{\alpha}) & \text{if}\quad x \in [0,\frac{1}{2}],\\
2x-1 & \text{if}\quad x \in (\frac{1}{2},1],
\end{cases} \quad \alpha >0.
\end{equation}
For these systems the intermittency is caused by a neutral fixed point at zero, which makes orbits spend long periods of time close to the origin, while behaving chaotically otherwise.

\medskip 
The statistical properties of the Manneville-Pomeau maps, LSV maps and other similar maps with a neutral fixed point are by now well understood. For instance, it is proven in \cite{pianigiani1980} that such maps admit an absolutely continuous invariant measure. Furthermore, polynomial bounds on the decay of correlations of such maps were obtained in \cite{hu04,LSV,Y99}, and it was shown in \cite{gouezel04} that the rate of the polynomial upper bound from \cite{Y99} is in fact sharp.

\medskip
More recently intermittency has been considered in random systems, where one has a family of maps $\{ T_j: [0,1] \to [0,1] \}_{j \in I}$, for some finite index set $I$, and a probabilistic rule to determine which map is applied at each time step. Random systems with a neutral fixed point were studied in e.g.~\cite{BBD14,BB16,KKV17,BBR19,BRS20,KMTV} and for the systems in \cite{BBD14,BB16,BBR19,BRS20} it was shown that this leads to intermittency.   
The recent papers \cite{AbbGhaHom,HP,Zeegers21} analysed so-called {\em critical intermittency}, where orbits are attracted towards a superstable fixed point under iterations of some of the maps in $\{T_j\}$, while the other maps map this fixed point to another, repelling fixed point. Figure~\ref{fig1}(a) shows an example with the two logistic maps $T_2(x) = 2x(1-x)$ and $T_4(x)=4x(1-x)$ that fits the framework of \cite{AbbGhaHom,Zeegers21}. In \cite{Zeegers21} the authors found a phase transition for the existence of a unique invariant probability measure that is absolutely continuous with respect to the Lebesgue measure, depending on the randomness parameter and the orders of attraction of the maps in $\{T_j\}$ that have a superstable fixed point.

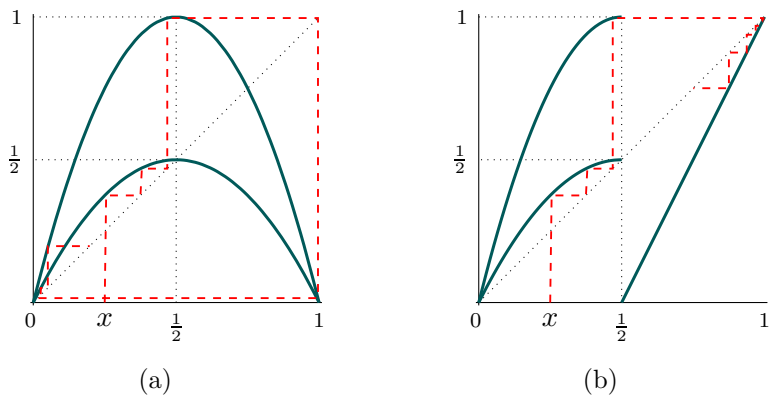
\begin{figure}[h] \label{fig1}
\centering
\subfigure[]
{
\begin{tikzpicture}[scale =3.8]
\draw(-.01,0)--(1.01,0)(0,-.01)--(0,1.01);
\draw[dotted](.5,0)--(.5,1)(0,1)--(.5,1)(.5,.5)--(0,.5);
\draw[dotted](0,0)--(1,1);
\draw[line width=.4mm, green!50!blue!70!black, smooth, samples =20, domain=0:1] plot(\x, { 4* \x * (1-\x)});
\draw[line width=.4mm, green!50!blue!70!black, smooth, samples =20, domain=0:1] plot(\x, { 2* \x * (1-\x)});
\draw[red, dashed, line width=.25mm](.25,0)--(.255,.375)--(.375,.375)--(.37995,.4688)--(.4688,.4688)--(.4688,.9961)--(.9961,.9961)--(.9961,.0155)--(.0155,.0155)--(.0155,.0262)--(.0262,.0262)--(0.0262,.051)--(0.051,.051)--(0.051,.1975)--(0.1975,.1975);

\node[below] at (.25,0){\small $x$};
\node[below] at (-.01,0){\tiny 0};
\node[below] at (1,0){\tiny 1};
\node[below] at (.5,0){\tiny $\frac12$};
\node[left] at (0,1){\tiny 1};
\node[left] at (0,.5){\tiny $\frac12$};
\end{tikzpicture}}
\hspace{1cm}
\subfigure[]
{
\begin{tikzpicture}[scale =3.8]
\draw(-.01,0)--(1.01,0)(0,-.01)--(0,1.01);
\draw[dotted](.5,0)--(.5,1)(0,1)--(.5,1)(.5,.5)--(0,.5);
\draw[dotted](0,0)--(1,1);
\draw[line width=.4mm, green!50!blue!70!black, smooth, samples =20, domain=0:0.5] plot(\x, { 4* \x * (1-\x)});
\draw[line width=.4mm, green!50!blue!70!black, smooth, samples =20, domain=0:0.5] plot(\x, { 2* \x * (1-\x)});
\draw[line width=.4mm, green!50!blue!70!black] (.5,0)--(1,1);
\draw[red, dashed, line width=.25mm](.25,0)--(.255,.375)--(.375,.375)--(.37995,.4688)--(.4688,.4688)--(.4688,.9961)--(.9961,.9961)--(.9961,.9922)--(.9922,.9922)--(.9922,.9844)--(.9844,.9844)--(0.9844,.9688)--(0.9688,.9688)--(0.9688,.9376)--(0.9376,.9376)--(0.9376,.8752)--(.8752,.8752)--(.8752,.7504)--(.7504,.7504);

\node[below] at (.25,0){\small $x$};
\node[below] at (-.01,0){\tiny 0};
\node[below] at (1,0){\tiny 1};
\node[below] at (.5,0){\tiny $\frac12$};
\node[left] at (0,1){\tiny 1};
\node[left] at (0,.5){\tiny $\frac12$};
\end{tikzpicture}}
\caption{Critical intermittency in the random system of (a) the logistic maps $T_2(x) = 2x(1-x)$ and $T_4(x)=4x(1-x)$, (b) the maps given by  \eqref{eq1u} with $r_g = 2$ and \eqref{eq1h} with $\ell_b = 2$. The dashed line indicates part of a random orbit of $x$.}
\label{fig:interm}
\end{figure}

\medskip
For those random systems from \cite{Zeegers21} that have an absolutely continuous invariant probability measure, one can naturally wonder about the mixing properties and decay of correlations. As is well known from e.g.~\cite{Keller92, Young92, Young98, BLS} the maps in $\{T_j\}$ without a superstable fixed point exhibit exponential decay of correlations. On the other hand, it has been conjectured in \cite{Zeegers21} that the random systems show polynomial decay of correlations. In this article we explore this question further, but instead of the random maps from \cite{Zeegers21} we work with adapted versions. The way in which we have adapted the systems from \cite{Zeegers21} very much resembles the way in which the LSV maps from \eqref{q:lsvmaps} are adaptations of the Manneville-Pomeau maps from \eqref{q:mpmaps}. We work with these adaptations because they allow us to build a suitable Young tower and use the corresponding results, while preserving the main dynamical properties of the maps $T_j$ from \cite{Zeegers21}.


\medskip 
We describe the systems we consider in more detail. Call a map $T_g: [0,1] \rightarrow [0,1]$ {\em good} if it is given by
\begin{align}\label{eq1u}
T_g(x) =  \begin{cases}
1 - 2^{r_g} (\frac{1}{2}-x)^{r_g} & \text{if}\quad x \in [0,\frac{1}{2}),\\
2x-1 & \text{if}\quad x \in [\frac{1}{2},1],
\end{cases}
\end{align}
for some $r_g \geq 1$ and denote the class of good maps by $\mathfrak G$. A map $T_b: [0,1] \rightarrow [0,1]$ is called {\em bad} if it is given by
\begin{align}\label{eq1h}
T_b(x) =  \begin{cases}
\frac{1}{2} - 2^{\ell_b-1} (\frac{1}{2}-x)^{\ell_b} & \text{if}\quad x \in [0,\frac{1}{2}),\\
2x-1 & \text{if}\quad x \in [\frac{1}{2},1],
\end{cases}
\end{align}
for some $\ell_b > 1$ and we denote the class of bad maps by $\mathfrak B$. See Figure \ref{fig1}(b) for an example of the good map with $r_g = 2$ and the bad map with $\ell_b = 2$, which on $[0,\frac{1}{2})$ are equal to the logistic maps $x \mapsto 4x(1-x)$ and $x \mapsto 2x(1-x)$, respectively. Note that if $r_g = 1$, then $T_g$ is equal to the doubling map. The random systems we consider in this article are the following. Let $\{ T_1,\ldots,T_N \} \subseteq \mathfrak G \cup \mathfrak B$ be a finite collection of good and bad maps. Write $\Sigma_G = \{1 \leq j \leq N\, :\,  T_j \in \mathfrak G\}$ and $\Sigma_B = \{1 \leq j \leq N\, :\,  T_j \in \mathfrak B\}$ for the index sets of the good and bad maps, respectively, and assume that $\Sigma_G,\Sigma_B \neq \emptyset$. Write $\Sigma = \{ 1, \ldots, N \} = \Sigma_G \cup \Sigma_B$. The skew product transformation  or {\em random map} $F$ is defined by
\begin{equation}\label{q:skewproduct}
 F:\Sigma^{\mathbb N} \times [0,1] \to \Sigma^{\mathbb N} \times [0,1], \, (\omega,x) \mapsto (\sigma \omega, T_{\omega_1}(x)),
 \end{equation}
where $\sigma$ denotes the left shift on sequences in $\Sigma^{\mathbb N}$.

\medskip
Throughout the text we use $\lambda$ to denote the one-dimensional Lebesgue measure. Let $\mathbf p = (p_j)_{j \in \Sigma}$ be a probability vector with strictly positive entries representing the probabilities with which we choose the maps $T_j$, $j \in \Sigma$. On $\Sigma^\mathbb N \times [0,1]$ we are interested in measures of the form $\mathbb P \times \mu$, where $\mathbb P$ is the $\mathbf p$-Bernoulli measure on $\Sigma^\mathbb N$ and $\mu$ is a Borel measure on $[0,1]$ absolutely continuous with respect to $\lambda$ and  satisfying
\begin{align*}
\sum_{j \in \Sigma} p_j \mu(T_j^{-1}A) = \mu(A), \qquad \text{for all Borel sets $A \subseteq [0,1]$}.
\end{align*}
In this case $\mathbb P \times \mu$ is an invariant measure for $F$ and we say that $\mu$ is a {\em stationary} measure for $F$. If, furthermore, $\mu$ is absolutely continuous with respect to $\lambda$, then we call $\mu$ an \emph{absolutely continuous stationary (acs)} measure for $F$.

\medskip 
Our first two main results establish that there exists a phase transition for the existence of an acs probability measure that is similar to the phase transition found in \cite{Zeegers21}. Set $\theta = \sum_{b \in \Sigma_B} p_b \ell_b$.

\begin{theorem}\label{result1a}
If $\theta \geq 1$, then $F$ admits no acs probability measure.
\end{theorem}

\begin{theorem}\label{result1b}
If $\theta < 1$, then $F$ admits a unique acs probablity measure $\mu$. Moreover, $F$ is mixing with respect to $\mathbb{P} \times \mu$ and the density $\frac{d\mu}{d\lambda}$ is bounded away from zero.
\end{theorem}


\medskip 
Our second set of main results involve the decay of correlations in case $\theta<1$.
Equip $\Sigma^{\mathbb{N}} \times [0,1]$ with the metric
\begin{equation}\label{q:metric}
d\big((\omega,x),(\omega',y)\big) = 2^{-\min\{i \in \mathbb{N} \, : \, \omega_i \neq \omega_i'\}} + |x-y|.
\end{equation}
For $\alpha \in (0,1)$, let $\mathcal{H}_{\alpha}$ be the class of $\alpha$-H\"older continuous functions on $\Sigma^{\mathbb{N}} \times [0,1]$, i.e.,
\[ \mathcal{H}_{\alpha} = \Big\{h : \Sigma^{\mathbb{N}} \times [0,1] \rightarrow \mathbb{R} \, \Big| \,  \sup \Big\{ \frac{|h(z_1)-h(z_2)|}{d(z_1,z_2)^{\alpha}} : z_1,z_2 \in \Sigma^{\mathbb{N}} \times [0,1], z_1 \neq z_2 \Big\} < \infty \Big\},\]
and set
\[ \mathcal{H} = \bigcup_{\alpha \in (0,1)} \mathcal{H}_{\alpha}.\]
For $f \in L^{\infty}(\Sigma^{\mathbb N} \times [0,1],\mu)$ and $h \in \mathcal{H}$ the {\em correlations} are defined by 
\begin{align*}
Cor_n(f,h) = \int_{\Sigma^{\mathbb N} \times [0,1]} f \circ F^n \cdot h \, d\mathbb{P} \times \mu - \int_{\Sigma^{\mathbb N} \times [0,1]} f \, d\mathbb{P} \times \mu  \int_{\Sigma^{\mathbb N} \times [0,1]} h\, d\mathbb{P} \times \mu.
\end{align*}
Set $\ell_{\max} = \max\{\ell_b: b \in \Sigma_B\}$ and
\begin{align}\label{eqgamma1}
\gamma_1 = \frac{\log \theta}{\log \ell_{\max}}.
\end{align}
The following result gives an upper bound on the decay of correlations of $F$ that is polynomial with rate arbitrarily close to $\gamma_1$.

\begin{theorem}\label{result2a} Assume that $\theta < 1$. If $\gamma \in (\gamma_1,0)$, $f \in L^{\infty}(\Sigma^{\mathbb{N}} \times [0,1], \mathbb{P} \times \mu)$ and $h \in \mathcal{H}$, then $|Cor_n(f,h)| = O(n^{\gamma})$.
\end{theorem}

For each $b \in \Sigma_B$ set
\begin{align}\label{eqpib}
\pi_b = \sum_{j \in \Sigma_B: \ell_j \geq \ell_b} p_j
\end{align}
and let
\begin{align}\label{eqgamma2}
\gamma_2 = 1 + \max\Big\{ \frac{\log \pi_b}{\log \ell_b}: b \in \Sigma_B\Big\}.
\end{align}
Note that
\[ \gamma_2 = \max\Big\{ \frac{\log (\pi_b \cdot \ell_b) }{\log \ell_b}: b \in \Sigma_B\Big\} \leq \max\Big\{ \frac{\log \theta }{\log \ell_b}: b \in \Sigma_B\Big\} = \gamma_1.\]
In our final result we show that, under additional assumptions on the parameters of the random systems, the class of observables $f \in L^{\infty}(\Sigma^{\mathbb N} \times [0,1],\mu)$ and $h \in \mathcal{H}$ contains functions for which the decay rate is at most polynomial with rate $\gamma_2$.

\begin{theorem}\label{result2b}
Assume that $\theta < 1$. Furthermore, assume that $\gamma_2 > \gamma_1 -1$ if $\gamma_1 < -1$ and $\gamma_2 > 2\gamma_1$ if $-1 \leq \gamma_1 < 0$. Let $f \in L^{\infty}(\Sigma^{\mathbb{N}} \times [0,1], \mathbb{P} \times \mu)$ and $h \in \mathcal{H}$ be such that both $f$ and $h$ are identically zero on $\Sigma^{\mathbb{N}} \times \big([0,\frac{1}{2}] \cup [\frac{3}{4},1]\big)$ and such that
\[ \int_{\Sigma^\mathbb N \times [0,1]} f \, d\mathbb{P} \times\mu \cdot \int_{\Sigma^\mathbb N \times [0,1]} h \, d\mathbb{P} \times \mu > 0.\]
Then
\begin{align}
|Cor_n(f,h)| = \Omega(n^{\gamma_2}).
\end{align}
\end{theorem}

In Section~\ref{sec4} we provide examples of values of $\ell_b$ and probability vectors $\mathbf p$ that satisfy the conditions of Theorem~\ref{result2b}.

\medskip
The proof of Theorem \ref{result1b} we present below also carries over to the case that $\Sigma_B = \emptyset$. Applying Theorem \ref{result1b} to the case that $\Sigma_B = \emptyset$ and $\Sigma_G = \{ g \}$ contains one element yields together with \cite[Theorem 1.5]{luzzatto13} the following result on the good maps $T_g \in \mathfrak G$.
\begin{cor}\label{thrm1.1}
For any $T_g: [0,1] \rightarrow [0,1] \in \mathcal G$ the following hold.
\begin{enumerate}
\item $T_g$ admits an invariant probability measure $\mu_g$ that is mixing and absolutely continuous with respect to Lebesgue measure $\lambda$.
\item There exists a constant $a > 0$ such that for each $f \in L^{\infty}(\mu_g)$ and each function $h: [0,1] \rightarrow \mathbb{R}$ of bounded variation we have
\begin{align*}
|Cor_{\mu_g,n}(f,h)| = O(e^{-an}),
\end{align*}
where
\[Cor_{\mu_g,n}(f,h) = \int f \circ T_g^n \cdot h \, d\mu_g - \int f d\mu_g \cdot \int h \, d\mu_g.\]
\end{enumerate}
\end{cor}

\medskip
This result is expected in view of the exponential decay of correlations found for the unimodal maps from \cite{Keller92, Young92, Young98}. We see that test functions that fall within the scope of both this corollary and Theorems \ref{result2a} and \ref{result2b} have exponential decay of correlations under a single good map while under the random system with $\Sigma_B \neq \emptyset$ they have polynomial decay of correlations.\footnote{Examples of such test functions are Lipschitz continuous  functions 
that vanish outside of $\Sigma^{\mathbb{N}} \times (\frac{1}{2},\frac{3}{4})$.} This indicates that a system of good maps loses its exponential decay of correlations when mixed with bad maps and instead adopts polynomial mixing rates. This is different from what has been observed for other random systems in e.g.~\cite{BBD14,BB16,BBR19}, where the random systems of LSV maps under consideration adopt the highest decay rate from the rates of the individual LSV maps present in the system.

\medskip 
This article is outlined as follows. In Section \ref{sec:mr} we introduce some notation and list some preliminaries. In Section \ref{sec3} we prove Theorem \ref{result1a}. The method of proof for this part is reminiscent of that in \cite{Zeegers21} in the sense that we introduce an induced system and apply Kac's Lemma on the first return times. In Section \ref{sec3} we also give estimates on the first return times and the induced map that we use later in Section \ref{sec4}. For Theorem~\ref{result1b} the approach from \cite{Zeegers21} no longer works, because we have introduced a discontinuity at $\frac12$. Instead we prove Theorem~\ref{result1b}, as well as Theorem \ref{result2a} and Theorem \ref{result2b}, by constructing a Young tower on the inducing domain from Section~\ref{sec3} and by applying the general theory from \cite{Y99} and \cite{gouezel04}. This is done in 
Section~\ref{sec4} and is inspired by the methods from \cite{BBD14,BB16}. Section \ref{sec5} contains some further results. More specifically, we show that for a specific class of test functions the upper bound from Theorem \ref{result2a} can be improved and we obtain a Central Limit Theorem. We end with some final remarks.

\section{Preliminaries}\label{sec:mr}
We use this section to introduce some notions and results from the literature. We present them in an adapted form, rephrasing them to our setting and only referring to the parts that are relevant for our purposes.

\subsection{Words and sequences}

As in the Introduction, let $T_1,\ldots,T_N \in \mathfrak G \cup \mathfrak B$ be a finite collection of good and bad maps, and write $\Sigma_G = \{1 \leq j \leq N\, :\,  T_j \in \mathfrak G\} \neq \emptyset$, $\Sigma_B = \{1 \leq j \leq N\, :\,  T_j \in \mathfrak B\} \neq \emptyset$ and $\Sigma = \{ 1, \ldots, N \} = \Sigma_G \cup \Sigma_B$.

\medskip
We use $\mathbf u \in \Sigma^n$ to denote a {\em word} $\mathbf u = u_1 \cdots u_n$. $\Sigma^0$ contains only the empty word, which we denote by $\epsilon$. We write $\Sigma^* = \bigcup_{n \ge 0} \Sigma^n$ for the collection of all finite words with digits from $\Sigma$. The concatenation of two words $\mathbf u, \mathbf v \in \Sigma^*$ is written as $\mathbf u \mathbf v$. For a word $\mathbf u \in \Sigma^*$ we use $|\mathbf u|$ to denote its length, so $|\mathbf u|=n$ if $\mathbf u \in \Sigma^n$. Similarly, we write $\Sigma_G^*$ and $\Sigma_B^*$ for the collections of finite words with elements in $\Sigma_G$ and $\Sigma_B$, respectively. On the space of infinite sequences $\Sigma^\mathbb N$ we consider the product topology obtained from the discrete topology on $\Sigma$. We use
\[ [\mathbf u] = [u_1 \cdots u_n] = \{\omega \in \Sigma^{\mathbb{N}}: \omega_1 = u_1, \ldots, \omega_n = u_n\}\]
to denote the cylinder set corresponding to a word $\mathbf u \in \Sigma^*$. For a probability vector $\mathbf p = (p_j)_{j \in \Sigma}$ and a word $\mathbf u \in \Sigma^n$ we write $p_{\mathbf u} = \prod_{i=1}^n p_{u_i} $ with $p_{\mathbf u}=1$ if $n=0$. Similarly, for $\mathbf b \in \Sigma_B^n$ we write $\ell_{\mathbf b} = \prod_{i=1}^n \ell_{b_i}$ with $\ell_{\mathbf b}=1$ if $n=0$.

\medskip
We use the following notation for compositions of $T_1,\ldots,T_N$. For each $\omega \in \Sigma^{\mathbb{N}}$ and each $n \in \mathbb{N}_0$ we write
\begin{equation}\label{q:concatenateT}
 T_{\omega_1 \cdots \omega_n }(x) = T_\omega^n(x) =  \begin{cases}  x, & \text{if}\quad n=0,\\
T_{\omega_n}\circ  T_{\omega_{n-1}}\circ \cdots \circ T_{\omega_1}(x), & \text{for } n\ge 1.
\end{cases}
\end{equation}
With this notation, we can write the iterates of the random system $F$ from \eqref{q:skewproduct} as
\[ F^n(\omega,x) = (\sigma^n \omega, T_{\omega}^n(x)).\]
We also use the notation from \eqref{q:concatenateT} for finite words $\mathbf u \in \Sigma^m$, $m \ge 1$, instead of sequences $\omega \in \Sigma^\mathbb N$ and $n \le m$.

\medskip
For ease of notation, we refer to the left branch of a map $T_j$ by $L_j$, i.e., for $x \in [0,\frac{1}{2})$ we write $L_j(x) = T_j(x)$, $j \in \Sigma$. We also use the notation from \eqref{q:concatenateT} in this situation, so for $\omega \in \Sigma^{\mathbb{N}}$, $n \in \mathbb{N}_0$ and $x \in [0,\frac{1}{2})$ such that $L_{\omega_j} \circ \cdots \circ L_{\omega_1}(x) \in [0,\frac{1}{2})$ for each $j=1,\ldots,n-1$ we write
\begin{align}\label{eq19d}
L_{\omega_1 \cdots \omega_n }(x) = L_\omega^n(x) =  \begin{cases}  x, & \text{if}\quad n=0,\\
L_{\omega_n}\circ  L_{\omega_{n-1}}\circ \cdots \circ L_{\omega_1}(x), & \text{for } n\ge 1.
\end{cases}
\end{align}
Also in this situation we use the same notation for finite words $\mathbf u \in \Sigma^m$, $m \ge 1$, and $n \le m$. We will use $R: [\frac12,1] \to [0,1], \, x \mapsto 2x-1$ to denote the right branche of the maps $T_j$.

\subsection{Invariant measures}
Each map $T_j$, $j \in \Sigma$, is Borel measurable and non-singular with respect to Lebesgue (that is, $\lambda(A) =0$ if and only if $\lambda(T_j^{-1}A) = 0$ for all Borel sets $A \subseteq [0,1]$). Therefore the following lemma applies to the skew product $F$ from \eqref{q:skewproduct}.

\begin{lemma}[\cite{Mor85}, see also Lemma 3.2 of \cite{Fro99}]\label{l:productmeasure}
Let $\mathbf p = (p_j)_{j \in \Sigma}$ be a positive probability vector and let $\mathbb P$ be the $\mathbf p$-Bernoulli measure on $\Sigma^{\mathbb{N}}$. Then the $\mathbb{P} \times \lambda$-absolutely continuous $F$-invariant finite measures are precisely the measures of the form $\mathbb{P} \times \mu$ where $\mu$ is finite and absolutely continuous w.r.t.~$\lambda$ and satisfies
\begin{align}\label{eqn9}
\sum_{j \in \Sigma} p_j \mu (T_j^{-1}A) = \mu (A) \qquad \text{for all Borel sets $A$}.
\end{align}
\end{lemma}

To find measures $\mu$ that satisfy \eqref{eqn9} and are absolutely continuous w.r.t.~$\lambda$ one can use the {\em Perron-Frobenius operator}. For maps $T: I \rightarrow I$ on an interval $I$ that are piecewise strictly monotone and $C^1$ the Perron-Frobenius operator $\mathcal P_T: L^1(I,\lambda) \rightarrow L^1(I,\lambda)$ is given by
\begin{equation}\label{q:pfd}
\mathcal P_T h (x) = \sum_{y \in T^{-1}\{x\}} \frac{h(y)}{|DT(y)|}.
\end{equation}
A non-negative function $\varphi \in L^1(I,\lambda)$ is a fixed point of $\mathcal P_T$ if and only if the measure $\mu$ given by $\mu(A) = \int_A \varphi \, d\lambda$ for each Borel set $A$ is an invariant measure for $T$. For a finite family $\{ T_j: [0,1] \to [0,1] \}_{j \in \Sigma}$ of good and bad maps and a positive probability vector $\mathbf p = (p_j)_{j \in \Sigma}$ the Perron-Frobenius operator $\mathcal{P}_{F,\mathbf p}: L^1(I,\lambda) \rightarrow L^1(I,\lambda)$ associated to the skew product $F$ and the vector $\mathbf p$ is given by
\begin{align}\label{eqn3.22}
\mathcal{P}_{F,\mathbf p}h(x) = \sum_{j \in \Sigma} p_j \mathcal P_{T_j} h (x),
\end{align}
where each $\mathcal{P}_{T_j}$ is as in \eqref{q:pfd}. A non-negative function $\varphi \in L^1([0,1],\lambda)$ is a fixed point of $\mathcal{P}_{F,\mathbf p}$ if and only if it provides an $F$-invariant measure $\mathbb{P} \times \mu$ where $\mathbb P$ denotes the $\mathbf p$-Bernoulli measure on $\Sigma^{\mathbb{N}}$ and $\mu$ is given by $\mu(A) = \int_A \varphi \, d\lambda$ for each Borel set $A$.

\subsection{Induced systems, Jacobian and Young towers}\label{s:youngtower}
Let $\mathcal F$ be the product $\sigma$-algebra on $\Sigma^\mathbb N \times [0,1]$ given by the Borel $\sigma$-algebra's on both coordinates. We call the set on which we will induce $Y \in \mathcal F$. It will be defined later. The \emph{first return time map} $\varphi: Y \rightarrow \mathbb{N}$ is defined as
\begin{align*}
\varphi(\omega,x) = \inf\{n \geq 1 : F^n(\omega,x) \in Y\}.
\end{align*}

\begin{lemma}[Kac's Lemma, see e.g.~1.5.5.~in \cite{Aar97}]\label{l:kac}
Let $m$ be an ergodic invariant and finite measure for $F$ on $(\Sigma^\mathbb N \times [0,1],\mathcal F)$ with $m(Y) > 0$. Then $\int_Y \varphi \, dm = m(Y)$.
\end{lemma}

A set $A \in \mathcal{F}$ is called an \emph{invertibility domain} of $F$ if the restriction $F|_A: A \rightarrow F(A)$ is bijective with measurable inverse. The following property gives conditions under which the Jacobian of $F$ with respect to $m$ exists.

\begin{prop}[Proposition 9.7.2 in \cite{viana16}]\label{prop2.1a} Suppose that $\mathcal I$ is a partition of $\Sigma^\mathbb N \times [0,1]$ by invertibility domains $A$ of $F$. Furthermore, suppose that $m$ on $(\Sigma^\mathbb N \times [0,1],\mathcal F)$ is a finite measure and that for each $A \in \mathcal I$ and $E \in \mathcal{F} \cap A$ it holds that $m(F(E)) = 0$ if $m(E) = 0$. Then there exists a ($m$-a.e.) unique nonnegative function $J_{m}F \in L^1(\Sigma^\mathbb N \times [0,1],m)$, called the \emph{Jacobian of $F$ w.r.t.~$m$}, such that
\begin{align*}
m(F(E)) = \int_E J_{m}F\, dm
\end{align*}
for each $E \in \mathcal F \cap A$ and $A \in \mathcal I$.
\end{prop}

We have the following change of variables formula.

\begin{lemma}[Lemma 9.7.4 in \cite{viana16}]\label{lemma2.1a} Under the assumptions of Proposition \ref{prop2.1a}, for each $A \in \mathcal I$ and $E \in \mathcal F \cap A$,
\[ \int_{F(E)} h \, dm = \int_E (h \circ F) J_{m} F \,  dm\]
for any measurable function $h: F(E) \rightarrow \mathbb{R}$ such that the integrals are defined (possibly $\pm \infty$).
\end{lemma}

In the situation of Kac's Lemma the {\em induced transformation} $F_Y: Y \to Y$ given by $F_Y(x) = F^{\varphi(x)}(x)$ is well defined almost everywhere. On the {\em inducing domain} $Y$ one can construct a Young tower. For the convenience of the reader we briefly recall this construction, which is outlined in general and in more detail in \cite{Y99}. Let $\mathcal P$ be a countable partition of $Y$ into measurable sets on which the first return time function $\varphi$ is constant, called a {\em first return time partition}. Let $\varphi_P$ be the constant value that $\varphi$ assumes on $P$. The tower is defined by
\begin{equation*}
\Delta = \{ (z,n) \in Y \times \{ 0,1,2 \ldots \} \, : \, n < \varphi(z)\}.
\end{equation*}
The $l$-th level of the tower is $\Delta_l := \Delta \cap \{ n = l\}$ and for each $P \in \mathcal P$ let $\Delta_{l,P} = \Delta_l \cap \{ z \in P \}$. We equip $\Delta$ with the Borel $\sigma$-algebra $\mathcal B$ and let $m$ be the unique measure on $\Delta$ that corresponds to $\mathbb P \times \lambda$ on each level of the tower. On $\Delta$ define the function
\begin{equation*}
G(z, n) = \begin{cases}
(z, n+1), & \text{if } n+1 < \varphi_Y(z),\\
(F_Y(z),0), & \text{otherwise}.
\end{cases}
\end{equation*}
Let the induced map $G^{\varphi}: \Delta_0 \to \Delta_0$ be defined by $G^{\varphi} (\upsilon) = G^{\varphi(z)}(\upsilon)$, where $z \in Y$ is such that $\upsilon = (z,0)$. We identify $G^{\varphi}$ with $F^{\varphi}$ by identifying $\Delta_0$ with $Y$ and using the correspondence $G^{\varphi}(z,0) = (F^{\varphi}(z),0)$. The {\em separation time} is defined for each $(z_1, l_1), (z_2,l_2) \in \Delta$ by $s((z_1,l_1), (z_2,l_2)) =0$ if $l_1 \neq l_2$ and otherwise letting $s((z_1,l_1), (z_2,l_2)) =s(z_1,z_2)$ be given by
\begin{equation}\label{q:septimetower}
s(z_1,z_2) = \inf \{ n \ge 0 \, : \, (G^{\varphi})^n (z_1,0), \, (G^{\varphi})^n (z_2,0) \text{ lie in distinct } \Delta_{0,P}, \, P \in \mathcal P\}.
\end{equation}

\vskip .2cm
The set-up from \cite{Y99} assumes the following conditions on $\Delta$ and $G$:
\begin{itemize}
\item[(t1)] $gcd\{ \varphi_P \} =1$;
\item[(t2)] All the sets in the construction above are measurable and $m(Y) < \infty$.
\item[(t3)] For each $P \in \mathcal P$ the top level $\Delta_{\varphi_P-1,P}$ above $P$ is mapped bijectively onto $\Delta_0 = Y \times \{0\}$ under the map $G$;
\item[(t4)] The partition $\eta = \{ \Delta_{l, P} \, : \, P \in \mathcal P, \,  0 \le l \le \varphi_P-1 \}$ generates $\mathcal B$.
\item[(t5)] The restrictions $G^\varphi|_{\Delta_{0,P}}: \Delta_{0,P} \to \Delta_0$ and their inverses are non-singular with respect to $m$, so that the Jacobian $JG^\varphi$ with respect to $m$ exists and is $>0$ $m$-a.e.
\item[(t6)] There are constants $C>0$ and $\beta\in (0,1)$ such that for each $P \in \mathcal P$ and all $(z_1,0), (z_2,0) \in \Delta_{0,P}$,
\[ \left| \frac{J_mG^\varphi (z_1,0)}{J_mG^\varphi (z_2,0)}-1 \right| \le C \beta^{s(G^\varphi (z_1,0), G^\varphi(z_2,0))}.\]
\end{itemize}

Under these conditions \cite{Y99} yields the following results on the existence of invariant measures for the map $G$ on the tower $\Delta$ and decay of correlations. Define the following function space on $\Delta$:
\begin{equation}\label{q:cbeta}
\mathcal{C}_{\delta} = \Big\{\hat f : \Delta \rightarrow \mathbb{R} \, \Big| \, \sup \Big\{ \frac{|\hat f(\upsilon_1)-\hat f(\upsilon_2)|}{\delta^{s(\upsilon_1,\upsilon_2)}} : \upsilon_1,\upsilon_2 \in \Delta, \upsilon_1 \neq \upsilon_2 \Big\} < \infty \Big\}.
\end{equation}
For $\upsilon \in \Delta$ let $\hat \varphi(\upsilon):= \inf \{ n \ge 0 \, : \, G^n (\upsilon) \in \Delta_0 \}$. 

\begin{theorem}[Theorem 1 and 3 from \cite{Y99}]\label{t:young1}
If (t1)--(t6) hold and $\int_Y \varphi \, dm < \infty$, then we have the following statements.\footnote{In \cite[Theorem 3]{Y99} statement (iv) above is only given for the case that $\delta = \beta$. Note however that the case $0 < \delta < \beta$
follows from $\mathcal{C}_{\delta} \subseteq \mathcal{C}_{\beta}$ and that the case $\beta < \delta < 1$ can be proven similarly
as \cite[Theorem 3]{Y99} by using the upper bound $C^+ \delta^{s(v_1,v_2)}$ for the left hand side of \eqref{q:densityC}.}
\begin{itemize}
\item[(i)] $G: \Delta \to \Delta$ admits an invariant probability measure $\nu$ that is absolutely continuous w.r.t.~$m$;
\item[(ii)] The density $\frac{d\nu}{dm}$ is bounded away from zero and is in $\mathcal C_\beta$ with $\beta$ as in (t6). Moreover, there is a constant $C^+>0$ such that for each $\Delta_{l,P}$ and each $\upsilon_1, \upsilon_2 \in \Delta_{l,P}$
\begin{equation}\label{q:densityC}
\left| \frac{\frac{d\nu}{ dm} (\upsilon_1)}{\frac{d\nu}{dm }(\upsilon_2)}-1 \right| \le C^+ \, \beta^{s(\upsilon_1,\upsilon_2)}.
\end{equation}
\item[(iii)] $G$ is exact, hence ergodic and mixing.
\item[(iv)] (Polynomial decay of correlations) If, moreover, $m(\{ \upsilon \in \Delta \, : \, \hat \varphi (\upsilon) >n \}) = O(n^{-\alpha})$ for some $\alpha >0$, then for all $\hat f \in L^\infty(\Delta,m)$, all $\delta \in (0,1)$ and all $\hat h \in \mathcal C_{\delta}$,
\[ \Big|\int_\Delta \hat f \circ G^n \cdot \hat h \, d\nu - \int_\Delta \hat f \, d\nu \int_\Delta \hat h \, d\nu \Big|  = O(n^{-\alpha}).\]
\end{itemize}
\end{theorem}

We will also use \cite[Theorem 6.3]{gouezel04} by Gou\"ezel which, when adapted to our setting, says the following.
\begin{theorem}[Theorem 6.3 from \cite{gouezel04}]\label{t:gouezel}
Let $\rho$ be an invariant and mixing probability measure for $F$ on $\Sigma^\mathbb N \times [0,1]$. Let $f \in L^{\infty} (\Sigma^\mathbb N \times [0,1], \mathbb P \times \mu)$ and  $h \in L^{1}(\Sigma^{\mathbb{N}} \times [0,1], \mathbb{P} \times \mu)$ be such that both $f$ and $h$ are identically zero on $(\Sigma^\mathbb N \times [0,1]) \backslash Y$. 
Assume that there is a $\delta \in (0,1)$ such that the following three conditions hold.
\begin{itemize}
\item[(g1)] There is a constant $C^*>0$ such that for each $n \ge 0$ and $z_1, z_2 \in \bigvee_{k=0}^{n-1} F^{-n}_Y \mathcal P$,
\[ \Big| \log \frac{J_\rho F^\varphi (z_1)}{J_\rho F^\varphi(z_2)} \Big| \le C^* \cdot \delta^n.\]
\item[(g2)] There is a $\zeta >1$ such that $\rho (\varphi >n) = O(n^{-\zeta})$.
\item[(g3)] $\sup \big\{ \frac{|h(z_1)-h(z_2)|}{\delta^{s(z_1, z_2)}} \, : \, z_1, z_2 \in Y, \, z_1 \neq z_2 \big\} < \infty$.
\end{itemize}
Then there is a constant $\tilde C >0$ such that 
\[ \Big| Cor_n (f,h) - \Big( \sum_{k > n} \rho (\varphi >k) \Big) \int f \, d\rho \int h \, d\rho \Big| \le \tilde C \cdot K_\zeta(n),\]
where
\[ K_\zeta(n) = \begin{cases}
n^{-\zeta}, & \text{if } \zeta >2,\\
\frac{\log n}{n^2}, & \text{if } \zeta =2,\\
n^{2-2\zeta}, & \text{if } \zeta \in (1,2).
\end{cases}\]
\end{theorem}


\subsection{Properties of good and bad maps}\label{subsec2.4}

Let $T:I \to \mathbb{R}$ be a $C^3$ map on an interval $I$. The {\em Schwarzian derivative} of $T$ at $x \in I$ with $DT(x) \neq 0$ is given by
\begin{equation*}
\mathbf ST(x) = \frac{D^3T(x)}{DT(x)} - \frac{3}{2} \Big(\frac{D^2 T(x)}{DT(x)}\Big)^2.
\end{equation*}
We say that $T$ {\em has non-positive Schwarzian derivative} if $DT(x) \neq 0$ and $\mathbf ST(x) \leq 0$ for all $x \in I$. One easily computes that each good map $T_g$ and each bad map $T_b$ has non-positive Schwarzian derivative when restricted to $[0, \frac12)$ or $[\frac12, 1]$. Since the composition of two maps with non-positive Schwarzian derivative again has non-positive Schwarzian derivative, 
any composition $T_{\omega}^n$ of good and bad maps has non-positive Schwarzian derivative. 

\medskip
The following well-known property of maps with non-positive Schwarzian will be of importance.

\medskip
\noindent {\bf Koebe Principle:} (see e.g.~\cite[Section 4.1]{dMvS93}) For each $c > 0$ there exists an $M^{(c)} > 0$ with the following property. Suppose that $J \subseteq I$ are two intervals and that $T: I \rightarrow \mathbb{R}$ has non-positive Schwarzian derivative. If both components of $T(I)\backslash T(J)$ have length at least $c \cdot \lambda(T(J))$, then
\begin{align}\label{eq2.5}
\Big|\frac{DT(x)}{DT(y)}-1\Big| \leq M^{(c)} \cdot \frac{|T(x)-T(y)|}{\lambda(T(J))}, \qquad \forall x,y \in J.
\end{align}
Note that the constant $M^{(c)}$ only depends on $c$ and not on the map $T$.

\medskip
For future reference we give a lemma on compositions of bad maps. It implies in particular that orbits starting in $[0,\frac{1}{2})$ converge superexponentially fast to $\frac{1}{2}$ under iterations of bad maps. Recall the notation of $L_{\mathbf b}$ from \eqref{eq19d}. 
\begin{lemma}\label{lemma2.1}
For each $x \in [0,\frac{1}{2})$ and $\mathbf b \in \Sigma_B^*$ we have
\begin{itemize}
\item[(i)] $L_{\mathbf b} (x) = \frac{1}{2}\big(1 -(1-2x)^{\ell_{\mathbf b}}\big)$, 
\item[(ii)] $L_{\mathbf b}^{-1} (x) = \frac{1}{2}\big(1 -(1-2x)^{\ell_{\mathbf b}^{-1}}\big)$.
\end{itemize}
\end{lemma}

\begin{proof}
Part (i) holds trivially if $|\mathbf b|=0$. Now suppose (i) holds for all $\mathbf b \in \Sigma_B^k$ for some $k \ge 0$. Let $\mathbf b b_{k+1} \in \Sigma_B^{k+1}$. Then $L_{\mathbf b}(x) \in [0,\frac{1}{2})$ and
\[ \begin{split}
L_{\mathbf b b_{k+1}} (x) =\ & \frac{1}{2}\Big(1 -(1-2L_{\mathbf b}(x))^{\ell_{b_{k+1}}}\Big)\\
=\ & \frac{1}{2}\Big(1-\big((1-2x)^{\ell_{\mathbf b}}\big)^{\ell_{b_{k+1}}}\Big) = \frac{1}{2}\Big(1 -(1-2x)^{\ell_{\mathbf b b_{k+1}}}\Big).
\end{split}\]
This proves (i), and (ii) follows easily from (i).
\end{proof}

We end this section with one more lemma on the properties of good and bad maps. For $g \in \Sigma_G$ the map $L_g: [0, \frac12) \to [0,1)$ is invertible and for $b \in \Sigma_B$ the map $L_b: [0, \frac12) \to [0, \frac12)$ is invertible. For all $j \in \Sigma$ the map $L_j$ is strictly increasing with continuous and decreasing derivative and, if $T_j$ is not the doubling map (i.e.~$j \in \Sigma_B$ or $j \in \Sigma_G$ with $r_g > 1$),
\[ \max_{x \in [0, \frac12)} DL_j (x) = DT_j(0) >1 \quad \text{and} \quad \lim_{x \uparrow \frac12} DL_j (x)  =0.\]
This allows us to define for each $g \in \Sigma_G$ with $r_g > 1$ the point $x_g$ as the point in $(0,\frac{1}{2})$ for which $DL_g(x_g) = 1$ and for each $b \in \Sigma_B$ the point $x_b$ as the point in $(0,\frac{1}{2})$ for which $DL_b(x_b) = 1$.

\begin{lemma}\label{lemma2.6d}
The following hold.
\begin{itemize}
\item[(i)] For each $g \in \Sigma_G$ it holds that $L_g^{-1} \big( \frac{1}{2} \big) \leq \frac{1}{4}$.
\item[(ii)] For each $g \in \Sigma_G$ with $r_g>1$ it holds that $L_g^{-1} \big( \frac{1}{2}\big)  < x_g$.
\item[(iii)] For each $b \in \Sigma_B$ it holds that $L_b^{-1} \big(\frac{1}{4}\big) < x_b$.
\end{itemize}
\end{lemma}

\begin{proof}
One can compute that
\begin{equation}\label{eq33f}
L_g^{-1} \Big(\frac{1}{2} \Big)  = \frac{1}{2}\big(1-2^{-1/r_g}\big)
\end{equation}
for all $ g  \in \Sigma_G$. If $r_g >1$, then
\[ x_g  =  \frac{1}{2} - \big(r_g \cdot 2^{r_g}\big)^{1/(1-r_g)}.\]
For each $b \in \Sigma_B$ we have
\begin{equation}\label{eq36f}
 L_b^{-1} \Big( \frac{1}{4} \Big) =  \frac{1}{2}\big(1-2^{-1/\ell_b}\big) \quad \text{ and } \quad  x_b  = \frac{1}{2}\big(1-\ell_b^{1/(1-\ell_b)}\big).
 \end{equation}
Since $r_g \geq 1$ and thus $2^{-1/r_g} \geq \frac{1}{2}$ for each $g \in \Sigma_G$, (i) follows. Furthermore, one can show that $2^{-1/x-1} > (x \cdot 2^x)^{1/(1-x)}$ and $2^{-1/x} > x^{1/(1-x)}$ hold for all $x > 1$, which give (ii) and (iii), respectively.
\end{proof}

\section{Inducing the random map on $(\frac{1}{2},\frac{3}{4})$} \label{sec3}




\subsection{The induced system and the infinite measure case} \label{subsec3.1}

Define
\[ \begin{split}
\tilde{\Omega} =\ & \{\omega \in \Sigma^{\mathbb{N}}: \omega_i \in \Sigma_G \text{ for infinitely many $i \in \mathbb{N}$}\},\\
Y =\ & \Big\{(\omega,x) \in \tilde{\Omega} \times \Big(\frac{1}{2},\frac{3}{4}\Big) : T_{\omega}^n(x) \neq \frac{1}{2} \text{ for all $n \in \mathbb{N}$}\Big\}.
\end{split}\]
The set $Y$, which equals $\Sigma^{\mathbb{N}} \times (\frac{1}{2},\frac{3}{4})$ up to a set of measure zero, will be our inducing domain. Recall the definition of the first return time function
\[ \varphi: Y \to \mathbb N, \, (\omega,x) \mapsto  \inf\{ n \geq 1\, :\,  F^n(\omega,x) \in Y\}.\]
For each $(\omega,x) \in Y$, the following happens under iterations of $F$. Firstly, $T_\omega(x) = R(x) \in (0, \frac12)$. By Lemma \ref{lemma2.6d} there is an $a > 1$ such that $DT_g(y) \geq a$ for all $y \in (0,L_g^{-1} \big(\frac12\big)]$ and $g \in \Sigma_G$. Combining this with Lemma \ref{lemma2.1} yields by definition of $Y$ that
\begin{equation}\label{q:to12}
\kappa(\omega,x) := \inf\Big\{ n \geq 1\, :\, T_{\omega}^n(x) > \frac12 \Big\} < \infty.
\end{equation}
Note that $\omega_{\kappa(\omega,x)} \in \Sigma_G$. Then, again since $(\omega,x) \in Y$, so $T_{\omega}^n(x) \neq \frac12$ for all $n \in \mathbb{N}$,
\begin{equation}\label{q:l}
l(\omega,x):= \inf \Big\{ n \ge 0 \, : \, T_\omega^{\kappa (\omega,x)+ n} \in \Big( \frac12, \frac34 \Big) \Big\} < \infty
\end{equation}
and $T^{\kappa(\omega,x)+l(\omega,x)}_\omega(x) = R^{l(\omega,x)} \circ T_\omega^{\kappa(\omega,x)}(x).$ Thus
\begin{equation}\label{q:phikappal}
\varphi(\omega,x) = \kappa(\omega,x) + l(\omega,x) < \infty
\end{equation}
for all $(\omega,x) \in Y$. We will first derive an estimate for $l(\omega,x)$.

\vskip .2cm
For each $g \in \Sigma_G$ let $J_g = (L_g^{-1}(\frac{1}{2}),\frac{1}{2})$. The intervals $J_g$ are such that if $T_\omega^n (x) \in J_g$ and $\omega_{n+1}=g$, then $T_\omega^{n+1}(x) > \frac12$. For each $b \in \Sigma_B$ let $J_b = [x_b,\frac{1}{2})$ with $x_b$ as defined above Lemma \ref{lemma2.6d}. Define 
\begin{align}\label{eq32a}
m(\omega,x) := \inf\{ n \geq 1\, :\, T_{\omega}^{n}(x) \in J_{\omega_{n+1}}\} < \kappa(\omega,x).
\end{align}
If $\omega_{m(\omega,x)+1} \in \Sigma_G$, then $T_{\omega}^{m(\omega,x)+1}(x) > \frac12$ by the definition of the intervals $J_g$. If $\omega_{m(\omega,x)+1} \in \Sigma_B$, then $T_{\omega}^{m(\omega,x)+1}(x) \in (\frac{1}{4},\frac{1}{2})$ by Lemma~\ref{lemma2.6d}(iii). We then see by Lemma~\ref{lemma2.6d}(i) that the number $m(\omega,x)$ is such that after this time any application of a good map will bring the orbit of $x$ in the interval $\big( \frac12,1\big)$. 
We have the following estimate on $l$.

\begin{lemma}\label{lemma3.2a}
Let $(\omega,x) \in Y$. Write $\mathbf d = \omega_{m(\omega,x)+1}\cdots \omega_{\kappa(\omega,x)-1} \in \Sigma_B^*$ and $g = \omega_{\kappa(\omega,x)} \in \Sigma_G$. Then
\[ l(\omega,x) \geq \ell_{ \mathbf d } r_g \frac{\log\big((1-2 \cdot T_{\omega}^{m(\omega,x)}(x) )^{-1}\big)}{\log 2}-2.\]
\end{lemma}

\begin{proof}
Set $y = T_{\omega}^{m(\omega,x)}(x)$. It follows from Lemma \ref{lemma2.1} that
\[ T_{\omega}^{\kappa(\omega,x)}(x) = L_g \circ L_{\mathbf d} (y) = 1-(1-2y)^{\ell_{\mathbf d} r_g}.\]
By definition of $R$, it follows that $l(\omega,x)$ is equal to the minimal $l \in \mathbb{N}_0$ such that
\begin{align}\label{q:2tol}
2^l \cdot (1-2y)^{\ell_{\mathbf d} r_g} > \frac{1}{4}.
\end{align}
Solving for $l$ gives the lemma.
\end{proof}

With this estimate we can now apply Kac's Lemma to prove Theorem \ref{result1a}.
\begin{proof}[Proof of Theorem \ref{result1a}]
Suppose $\mu$ is an acs probability measure for $F$. We first show that $\mu((\frac{1}{2},\frac{3}{4})) > 0$. Since $F^2(Y)$ equals $\Sigma^{\mathbb{N}} \times [0,1]$ up to some set of $\mathbb{P} \times \lambda$-measure zero and by \eqref{q:phikappal} all $(\omega,x) \in Y$ have a finite first return time to $Y$ under $F$, it follows that $\Sigma^{\mathbb{N}} \times [0,1]$ equals $\bigcup_{n=0}^{\infty} F^{-n} Y$ up to some set of $\mathbb{P} \times \lambda$-measure zero. Since $\mu$ is absolutely continuous with respect to $\lambda$, we obtain that
\[ 1 = \mathbb{P} \times \mu\big(\Sigma^{\mathbb{N}} \times [0,1] \big) \leq \sum_{n=0}^{\infty} \mathbb{P} \times \mu(F^{-n} Y) = \sum_{n=0}^{\infty} \mu\Big(\Big(\frac{1}{2},\frac{3}{4}\Big)\Big),\] from which we see that we indeed must have $\mu((\frac{1}{2},\frac{3}{4})) > 0$. Using the continuity of the measure, we know there exists an $a > \frac{1}{2}$ such that $\mu((a,\frac{3}{4})) > 0$. Note that for all $(\omega,x) \in \Sigma^{\mathbb{N}} \times (a,\frac{3}{4}) \cap Y$ we have
\begin{equation}\label{q:T14}
T_\omega^{m(\omega,x)} (x) \geq R(a).
\end{equation}
Fix a $g \in \Sigma_G$ and consider the subsets $A_{\mathbf b} = [\mathbf b g] \times ( a, \frac 34 ) \cap Y$, $\mathbf b \in \Sigma_B^*$, of $\Sigma^{\mathbb{N}} \times (a,\frac{3}{4}) \cap Y$. Set $p_B = \sum_{b \in \Sigma_B} p_b$. Then by \eqref{q:phikappal}, \eqref{q:T14} and Lemma \ref{lemma3.2a} we get
\begin{equation}\begin{split}\label{eq45a}
\int_Y \varphi \, d\mathbb{P} \times \mu \ge \ & \sum_{\mathbf b \in \Sigma_B^*} \int_{A_{ \mathbf b }} l(\omega,x) \, d\mathbb{P} \times \mu \\
\ge \ & \mu\Big(\Big(a,\frac{3}{4}\Big)\Big) \cdot \sum_{\mathbf b \in \Sigma_B^*} p_{\mathbf b g} \cdot \Big(\ell_{\mathbf b} r_g \cdot \frac{\log\big((1-2\cdot R(a))^{-1}\big)}{\log 2}-2\Big) \\
=\ & \mu\Big(\Big(a,\frac{3}{4}\Big)\Big) \cdot p_g \cdot \Big( r_g \cdot \frac{\log\big((1-2\cdot R(a))^{-1}\big)}{\log 2} \sum_{\mathbf b \in \Sigma_B^*} p_{\mathbf b}\ell_{\mathbf b} -2 \sum_{\mathbf b \in \Sigma_B^*} p_{\mathbf b} \Big) \\
=\ & M_1 \cdot \sum_{k \ge 0} \theta^k - M_2 \cdot \frac{1}{1-p_B}, 
\end{split}\end{equation}
with $M_1 = \mu((a,\frac{3}{4})) \cdot p_g r_g \cdot \frac{\log((1-2\cdot R(a))^{-1})}{\log 2}> 0$ and $M_2 = 2\mu((a,\frac{3}{4})) \cdot p_g$. It follows from the Ergodic Decomposition Theorem, see e.g.~Theorem 6.2 in \cite{einsiedler11}, that there exists a probability space $(E,\mathcal{E},\nu)$ and a measurable map $e \mapsto \mu_e$ with $\mu_e$ an $F$-invariant ergodic probability measure for $\nu$-a.e.~$e \in E$, such that
\[ \int_Y \varphi \, d\mathbb{P} \times \mu = \int_E \Big(\int_Y \varphi d\mu_e\Big) d\nu(e).\]
Combining this with \eqref{eq45a} we see that if $\theta \geq 1$, then there exists an $F$-invariant ergodic probability measure $\tilde{\mu}$ such that
\[ \int_Y \varphi d\tilde{\mu} = \infty, \]
in contradiction with Lemma \ref{l:kac}.
\end{proof}

\subsection{Estimates on the first return time $\varphi$}\label{subsec3.2}
From now on we only consider the case $\theta < 1$. We first define a first return time partition for $F$ to $Y$. For any $u \in \Sigma$, $g \in \Sigma_G$ and $\mathbf s, \mathbf w \in \Sigma^*$ write
\[ P_{u \mathbf s g \mathbf w} := ([u\mathbf s g \mathbf w] \cap \tilde \Omega ) \times (R^{|\mathbf w|} \circ L_g \circ L_{\mathbf s} \circ R)^{-1} \Big( \frac12, \frac34 \Big)\]
and define the collection of sets
\begin{equation}\label{q:frtp}
\mathcal P = \Big\{ P_{u \mathbf s g \mathbf w} \, : \, u \in \Sigma,\, g \in \Sigma_G, \, \mathbf s, \mathbf w \in \Sigma^* \Big\}.
\end{equation}
See Figure \ref{fig:interm2} for an illustration.

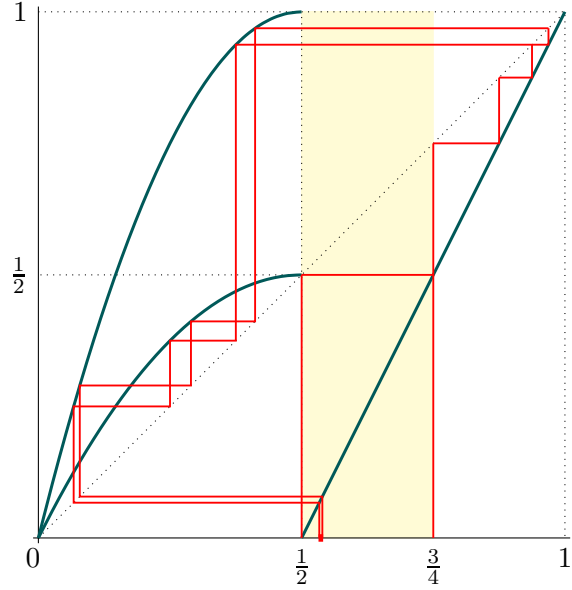
\begin{figure}[h] \label{fig2}
\centering
\begin{tikzpicture}[scale =7]
\filldraw[fill=yellow!20, draw=yellow!20] (.5,0) rectangle (.75,1);
\draw(-.01,0)--(1.01,0)(0,-.01)--(0,1.01);
\draw[dotted](.5,0)--(.5,1)(0,1)--(.5,1)(.5,.5)--(0,.5);
\draw[dotted](0.5,1)--(1,1)--(1,0);
\draw[dotted](0,0)--(1,1);
\draw[line width=.4mm, green!50!blue!70!black, smooth, samples =20, domain=0:0.5] plot(\x, { 4* \x * (1-\x)});
\draw[line width=.4mm, green!50!blue!70!black, smooth, samples =20, domain=0:0.5] plot(\x, { 2* \x * (1-\x)});
\draw[line width=.4mm, green!50!blue!70!black] (.5,0)--(1,1);
\draw[red, line width=.25mm](0.5,0)--(0.5,0.5);
\draw[red, line width=.25mm](0.75,0)--(0.75,0.5);
\newcommand{\zxxxx}{0.5}
\foreach \i in {1,...,4}{%
       \pgfmathparse{0.5*\zxxxx+0.5}
       \let\y\pgfmathresult
       \draw[red, line width=.25mm](\zxxxx,\zxxxx)--(\y,\zxxxx)--(\y,\y);
       \global\let\zxxxx\y
}
\draw[red, line width=.25mm](0.53350,0)--(0.53350,0.06699)--(0.06699,0.06699)--(0.06699,0.25)--(0.25,0.25)--(0.25,0.375)--(0.375,0.375)--(0.375,0.9375)--(0.9375,0.9375);
\draw[red, line width=.25mm](0.53931,0)--(0.53931,0.07862)--(0.07862,0.07862)--(0.07862,0.28977)--(0.28977,0.28977)--(0.28977,0.41161)--(0.41161,0.41161)--(0.41161,0.96875)--(0.96875,0.96875);

%
%

\node[below] at (-.01,0){\small 0};
\node[below] at (1,0){\small 1};
\node[below] at (.5,0){\small $\frac12$};
\node[left] at (0,1){\small 1};
\node[left] at (0,.5){\small $\frac12$};
\node[below] at (.75,0){\small $\frac34$};
\draw[red, line width=1mm](0.5320,0)--(0.5408,0);
\end{tikzpicture}
\caption{Example of a first return time partition element $P_{u \mathbf s g \mathbf w}$ with $|\mathbf s| =2$ and $|\mathbf w| = 3$ projected onto $[0,1]$. The yellow area indicates the inducing domain $Y$ projected onto $[0,1]$.}
\label{fig:interm2}
\end{figure}

\begin{prop}\label{p:frtp}
The collection $\mathcal P$ is a {\em first return time partition} of $F$ to $Y$ and for all $(\omega,x)$ in a set $P_{u\mathbf s g \mathbf w}$ it holds that $\kappa(\omega,x) = 2 + |\mathbf s|$ and $l(\omega,x) = |\mathbf w|$, so
\begin{align}\label{eqvarphi2}
\varphi(\omega,x) = 2 + |\mathbf s| + |\mathbf w|.
\end{align}
\end{prop}

\begin{proof}
Let $(\omega,x) \in Y$. Since $\kappa(\omega,x), l(\omega,x) < \infty$ it is clear that we can find suitable $u \in \Sigma$, $g \in \Sigma_G$ and $\mathbf s, \mathbf w \in \Sigma^*$ so that $(\omega,x) \in P_{u\mathbf s g \mathbf w}$, so $\mathcal P$ covers $Y$. Now fix a set $P_{u\mathbf s g \mathbf w} \in \mathcal P$. By the definition of the set $(R^{|\mathbf w|} \circ L_g \circ L_{\mathbf s} \circ R)^{-1} ( \frac12, \frac34)$ one has for any $(\omega,x) \in P_{u\mathbf s g \mathbf w}$ that
\[ \begin{split}
T_\omega^n(x) & < \textstyle \frac12, \quad 1 \le n \le 1 + |\mathbf s|,\\
T_\omega^n(x) & >  \textstyle \frac34, \quad 2 + |\mathbf s| \le n \le 1 + |\mathbf s| +|\mathbf w|,\\
T_\omega^{2 + |\mathbf s| + |\mathbf w|} (x) & \in \textstyle (\frac12, \frac34).
\end{split}\]
Hence, $\kappa(\omega,x) = 2 + |\mathbf s|$ and $l(\omega,x) = |\mathbf w|$ for each $(\omega,x) \in P_{u \mathbf s g \mathbf w}$, and \eqref{eqvarphi2} follows from \eqref{q:phikappal}. From this we immediately obtain that the sets in $\mathcal P$ are disjoint. To see this, suppose there are two different sets $P_{u\mathbf s g \mathbf w}, P_{\tilde u \tilde{\mathbf s}\tilde g \tilde{\mathbf w}} \in \mathcal P$ with $P_{u\mathbf s g \mathbf w} \cap P_{\tilde u \tilde{\mathbf s}\tilde g \tilde{\mathbf w}} \neq \emptyset$ and let $(\omega,x) \in P_{u\mathbf s g \mathbf w}\cap P_{\tilde u \tilde{\mathbf s} \tilde g \tilde{\mathbf w}}$. Then without loss of generality we can assume that $[\tilde u \tilde{\mathbf s}\tilde g \tilde{\mathbf w}] \subseteq [u\mathbf s g \mathbf w]$, where the inclusion is strict. But this would give that
\[ \varphi(\omega,x) = 2 + |\mathbf s|+ |\mathbf w| < 2 + |\tilde{\mathbf s}|+ |\tilde{\mathbf w}| = \varphi(\omega,x),\]
a contradiction.
\end{proof}

For the estimates we give below, we split the word $\mathbf s$ into two parts $\mathbf s = \mathbf v \mathbf b$, where $\mathbf b$ specifies the string of bad digits that immediately precedes $\omega_{\kappa(\omega,x)}$. In other words, if we write
\[ A_G = \{ \mathbf v \in \Sigma^* \, : \,  v_{|\mathbf v| \in \Sigma_G}\}\]
for the set of words that end with a good digit, then for any $\mathbf s \in \Sigma^*$ there are unique $\mathbf v \in A_G$ and $\mathbf b \in \Sigma_B^*$ such that $\mathbf s = \mathbf v \mathbf b$. Recall $\gamma_1$ and $\gamma_2$ from \eqref{eqgamma1} and \eqref{eqgamma2}, respectively. In the remainder of this subsection we prove the following result.

\begin{prop}\label{prop3.3b}
Suppose $\theta < 1$. Then the following statements hold.
\begin{itemize}
\item[(i)] $\int_Y \varphi \, d\mathbb{P} \times \lambda < \infty$.
\item[(ii)] $\mathbb{P} \times \lambda(\varphi > n) = O(n^{\gamma-1})$ for any $\gamma \in (\gamma_1,0)$.
\item[(iii)] $\mathbb{P} \times \lambda(\varphi > n) = \Omega(n^{\gamma_2-1})$.
\end{itemize}
\end{prop}

For the proof of Proposition \ref{prop3.3b} we will first prove three lemma's. Write
\begin{align}\label{eqp6a}
s = \Big(\min\Big\{ \Big\{DL_g\Big(L_g^{-1} \Big(\frac{1}{2}\Big) \Big) \, :\, g \in \Sigma_G\Big\} \cup \Big\{DL_b\Big(L_b^{-1} \Big( \frac{1}{4}\Big) \Big)\, :\,  b \in \Sigma_B\Big\}\Big)^{-1}.
\end{align}
The number $\frac1s$ will serve below as a lower bound on the derivative of the maps $T_j$ in some situation. Using Lemma \ref{lemma2.6d}, we see that $s \in (0,1)$.

\begin{lemma}\label{lemma3.4c}
For each $n \in \mathbb{N}$ we have
\[ \begin{split}
 \mathbb{P} \times \lambda( \varphi > n ) \leq \ & \frac14 \cdot \sum_{j=0}^{\infty} s^j \sum_{k=0}^{\infty} \sum_{\mathbf b \in \Sigma_B^k} \frac{p_{\mathbf b}}{2^{\max(n-1-j-k,1)\ell_{\mathbf b}^{-1} r_{\max}^{-1}}}, \\
 \mathbb{P} \times \lambda( \varphi > n ) \geq \ & \frac14 \cdot \min\{p_g\, :\,  g \in \Sigma_G\} \cdot \sum_{k=0}^{\infty} \sum_{\mathbf b \in \Sigma_B^k} \frac{p_{\mathbf b}}{2^{\max(n-1-k,1)\ell_{\mathbf b}^{-1} r_{\min}^{-1}}}.
\end{split}\]
\end{lemma}

\begin{proof}
For $P = P_{u \mathbf v \mathbf b g \mathbf w} \in \mathcal P$ we know from Proposition~\ref{p:frtp} that the first return time is constant on $P$ and equal to $\varphi_P = 2+ |\mathbf v| + |\mathbf b| + |\mathbf w|$. Let $n \in \mathbb{N}$. Then
\[ \mathbb{P} \times \lambda (\varphi > n) = \sum_{P: \varphi_P > n} \mathbb{P} \times \lambda (P).\]
To obtain the desired lower bound on $\mathbb{P} \times \lambda( \varphi > n )$ we only consider those $P=P_{u \mathbf v \mathbf b g \mathbf w} \in \mathcal P$ where $\mathbf v = \epsilon$ is the empty word. From Lemma \ref{lemma2.1} we get
\begin{align}\label{eq60c}
( R^{|\mathbf w|}\circ L_g \circ L_{\mathbf b})^{-1} \Big(\frac{1}{2},\frac{3}{4}\Big) = \Big(\frac{1}{2}\Big(1-\frac{1}{2^{(|\mathbf w|+1)\ell_{\mathbf b}^{-1} r_g^{-1}}}\Big),\frac{1}{2}\Big(1-\frac{1}{2^{(|\mathbf w|+2)\ell_{\mathbf b}^{-1} r_g^{-1}}}\Big)\Big).
\end{align}
Since $R$ has derivative 2, we then have
\[
\lambda\Big((R^{|\mathbf w|} \circ L_g \circ L_{\mathbf b} \circ R)^{-1} \Big(\frac{1}{2},\frac{3}{4}\Big)\Big) =  \frac{1}{4} \Big( \Big(\frac{1}{2^{|\mathbf w|+1}}\Big)^{\ell_{\mathbf b}^{-1 }r_g^{-1}} -  \Big(\frac{1}{2^{|\mathbf w|+2}}\Big)^{\ell_{\mathbf b}^{-1}r_g^{-1}}\Big)
\]
and thus
\begin{align}
\mathbb{P} \times \lambda (\varphi > n)  \ge \ & \sum_{u \in \Sigma} \sum_{g \in \Sigma_G} \sum_{\mathbf b \in \Sigma_B^*} \sum_{\stackrel{\mathbf w \in \Sigma^*:}{|\mathbf w| \ge \max \{ 0, n-2-|\mathbf b|\}} }  
\frac{\mathbb P ([u\mathbf b g \mathbf w])}{4}\Big( \frac{1}{2^{(|\mathbf w|+1)\ell_{\mathbf b}^{-1 }r_g^{-1}}} -  \frac{1}{2^{(|\mathbf w|+2)\ell_{\mathbf b}^{-1}r_g^{-1}}}\Big) \nonumber\\
= \ &  \sum_{g \in \Sigma_G}\frac{p_g }{4}  \sum_{k=0}^{\infty}  \sum_{\mathbf b \in \Sigma_B^k}  p_{\mathbf b} \sum_{l=\max \{ 0, n-2-k\}}^{\infty} \Big( \frac{1}{2^{(l+1)\ell_{\mathbf b}^{-1}r_g^{-1}}} -  \frac{1}{2^{(l+2)\ell_{\mathbf b}^{-1}r_g^{-1}}}\Big) \label{eq80v} \\
 \ge \ & \frac14 \cdot \min\{p_g\, :\,  g \in \Sigma_G\}   \cdot \sum_{k=0}^{\infty} \sum_{\mathbf b \in \Sigma_B^k} \frac{p_{\mathbf b}}{2^{\max \{1, n-1-k\}\ell_{\mathbf b}^{-1} r_{\min}^{-1}}}. \nonumber
\end{align}

\vskip .2cm
For the upper bound, we look for the smallest derivative to determine the length of $(R^{|\mathbf w|} \circ L_g \circ L_{\mathbf b} \circ L_{\mathbf v} \circ R)^{-1} (\frac{1}{2},\frac{3}{4})$. If $\mathbf v = \epsilon$, then from above we see
\[ \lambda\Big( (R^{|\mathbf w|} \circ L_g \circ L_{\mathbf b} \circ R)^{-1} \Big(\frac{1}{2},\frac{3}{4}\Big)\Big) =  \frac{s^{|\mathbf v|}}{4} \Big( \Big(\frac{1}{2^{|\mathbf w|+1}}\Big)^{\ell_{\mathbf b}^{-1 }r_g^{-1}} -  \Big(\frac{1}{2^{|\mathbf w|+2}}\Big)^{\ell_{\mathbf b}^{-1}r_g^{-1}}\Big).\]
On the other hand, if $\mathbf v=v_1 \cdots v_j$ with $j \ge 1$, then $v_j \in \Sigma_G$. We have
\[ ( R^{|\mathbf w|}\circ L_g \circ L_{\mathbf b})^{-1} \Big(\frac{1}{2},\frac{3}{4}\Big) \subseteq \Big( 0, \frac12\Big).\]
As follows from before $s$ from \eqref{eqp6a} represents the smallest possible shrinkage factor when applying $L_{v_j}^{-1}$. If $j \ge 2$, then by Lemma~\ref{lemma2.6d}(i) we have $L_{v_{j-1}}^{-1}(L_{v_j}^{-1} (\frac{1}{2})) \leq L_{v_{j-1}}^{-1}(\frac{1}{4})$. Hence, $s^{-|\mathbf v|}$ is a lower bound for the derivative of $L_{\mathbf v}$ for any $\mathbf v \in A_G$ on $( R^{|\mathbf w|}\circ L_g \circ L_{\mathbf b})^{-1} (\frac{1}{2},\frac{3}{4})$. It then follows from \eqref{eq60c} that
\[ \lambda\Big( (R^{|\mathbf w|} \circ L_g \circ L_{\mathbf b}\circ  L_{\mathbf v} \circ R)^{-1} \Big(\frac{1}{2},\frac{3}{4}\Big)\Big) \le  \frac{s^{|\mathbf v|}}{4} \Big( \Big(\frac{1}{2^{|\mathbf w|+1}}\Big)^{\ell_{\mathbf b}^{-1 }r_g^{-1}} -  \Big(\frac{1}{2^{|\mathbf w|+2}}\Big)^{\ell_{\mathbf b}^{-1}r_g^{-1}}\Big).\]
Writing $f(n,j,k) = \max(n-2-j-k,0)$, we thus obtain that
\begin{align}
\mathbb{P} \times \lambda (\varphi > n) & \le \sum_{g \in \Sigma_G}\sum_{\mathbf v \in A_G}  \sum_{\mathbf b \in \Sigma_B^*}  \sum_{\stackrel{\mathbf w \in \Sigma^*:}{|\mathbf w| \ge f(n, |\mathbf v|, |\mathbf b|)}} \frac{p_{\mathbf v \mathbf b g \mathbf w} s^{|\mathbf v|}}{4} \Big( \frac{1}{2^{(|\mathbf w|+1)\ell_{\mathbf b}^{-1 }r_g^{-1}}} -  \frac{1}{2^{(|\mathbf w|+2)\ell_{\mathbf b}^{-1 }r_g^{-1}}} \Big) \nonumber \\
& \le \frac14 \sum_{j=0}^{\infty} s^j \sum_{k=0}^{\infty}  \sum_{g \in \Sigma_G} \sum_{\mathbf b \in \Sigma_B^k} p_{\mathbf b} \sum_{l=f(n,j,k)}^{\infty} \Big( \frac{1}{2^{(l+1)\ell_{\mathbf b}^{-1}r_g^{-1}}} -  \frac{1}{2^{(l+2)\ell_{\mathbf b}^{-1}r_g^{-1}}}\Big) \label{eq77t} \\
& \le \frac14 \sum_{j=0}^{\infty} s^j \sum_{k=0}^{\infty} \sum_{\mathbf b \in \Sigma_B^k} \frac{p_{\mathbf b}}{2^{(f(n,j,k)+1)\ell_{\mathbf b}^{-1} r_{\max}^{-1}}}. \nonumber
\end{align}
This gives the result.
\end{proof}

The next lemma gives estimates for the last part of the expression on the right hand side of the first inequality from Lemma~\ref{lemma3.4c} for an initial range of values of $j$. The number of values $j$ for which we obtain an upper bound for the double sum grows logarithmically with $n$.

\begin{lemma}\label{lemma3.5c}
Let $\gamma \in (\gamma_1,0)$, and for each $n \in \mathbb{N}$ define $j(n) = \floor{\frac{(\gamma-1) \log n}{\log s}}$. Then there exist $C_1 > 0$ and $n_1 \in \mathbb{N}$ such that for all integers $n \geq n_1$ and $j =0,1,\ldots,j(n)$ we have
\begin{align}\label{q:sumk}
\sum_{k=0}^{\infty} \sum_{\mathbf b \in \Sigma_B^k}  \frac{p_{\mathbf b}}{2^{\max(n-1-j-k,1)\ell_{\mathbf b}^{-1} r_{\max}^{-1}}} \leq C_1 \cdot n^{\gamma-1}.
\end{align}
\end{lemma}

\begin{proof}
We will split the sum over $k$ in \eqref{q:sumk} into three pieces: $0 \le k \le k_1(n)$, $k_1(n) +1 \le k \le k_2(n)$ and $k > k_2(n)$. To define $k_1(n)$, let $a = \frac{\gamma \log \ell_{\max}}{\log \theta}\in (0,1)$. Then for each $n \in \mathbb{N}$ set $k_1(n)  = \big\lfloor \frac{\log(n^a)}{\log \ell_{\max}} \big\rfloor = \floor{\frac{\gamma \log n}{\log \theta}}$. The values $a$ and $k_1(n)$ are such that $\theta^{k_1(n)+1} \leq n^{\gamma}$ and $\ell_{\max}^{k_1(n)} \leq n^a$. Since $j(n)+ k_1(n) = O(\log n)$, we can find an $N_0 \in \mathbb{N}$ and a constant $K_1 > 0$ such that for all integers $n \geq N_0$ we have $n-1-j(n)-k_1(n) > 1$ and $(n-1-j(n)-k_1(n)) \cdot r_{\max}^{-1} \geq K_1 \cdot n$. Then, for all $n \geq N_0$, $0 \le j \le j(n)$ and $0 \le k \le k_1(n)$,
\[ \max(n-1-j-k,1) \ge \max (n-1-j(n)-k_1(n),1) = n-1-j(n)-k_1(n)\]
and $\ell_{\mathbf b} \le \ell_{\max}^{k_1(n)} \le n^a$ for all $\mathbf b \in \Sigma_B^k$. Setting $p_B = \sum_{b \in \Sigma_B} p_b$ as before, this together gives
\begin{equation}\begin{split}
\sum_{k=0}^{k_1(n)} \sum_{\mathbf b \in \Sigma_B^k} \frac{ p_{\mathbf b}}{2^{\max(n-1-j-k,1)\ell_{\mathbf b}^{-1} r_{\max}^{-1}}}  & \leq \sum_{k=0}^{\infty} \frac{p_B^k}{2^{(n-1-j(n)-k_1(n))n^{-a} r_{\max}^{-1}}} \\
& \leq \frac{2^{-K_1 \cdot n^{1-a}}}{1-p_B}. \label{eq68c}
\end{split}\end{equation}
Secondly, for each $n \in \mathbb{N}$ set $k_2(n) = \ceil{\frac{1}{2}(n-1-j(n))}$ and take an integer $N_1 \geq N_0$ and constant $K_2 > 0$ such that for all integers $n \geq N_1$ we have $k_2(n) \geq k_1(n)+1$ and $\frac{1}{2}(n-1-j(n))-1 \geq K_2 \cdot n$. Noting that for each $d > 1$ the function $f$ on $\mathbb{R}$ given by $f(x) = \frac{x}{d^x}$ has maximal value $\frac{1}{e \log d} < \frac1{\log d}$, we obtain for all integers $n \geq N_1$,
\begin{equation}\begin{split}
\sum_{k=k_1(n)+1}^{k_2(n)} \sum_{\mathbf b \in \Sigma_B^k}  \frac{p_{\mathbf b}}{2^{\max(n-1-j-k,1)\ell_{\mathbf b}^{-1} r_{\max}^{-1}}}  & = \sum_{k=k_1(n)+1}^{k_2(n)} \sum_{\mathbf b \in \Sigma_B^k} p_{\mathbf b} \ell_{\mathbf b} \frac{\ell_{\mathbf b}^{-1} }{(2^{(n-1-j-k)r_{\max}^{-1}})^{\ell_{\mathbf b}^{-1}}} \\
&\leq \sum_{k=k_1(n)+1}^{k_2(n)}\theta^k \frac{r_{\max}}{(n-1-j-k)\log 2} \\
&\leq \frac{\theta^{k_1(n)+1}}{1-\theta} \cdot \frac{r_{\max}}{(n-1-j(n)-k_2(n))\log 2} \\
&\leq \frac{r_{\max}}{(1-\theta)K_2 \log 2} \cdot n^{\gamma-1}. \label{eq68d}
\end{split}\end{equation}
Finally, for each $n \geq N_1$ we have
\begin{align}
\sum_{k=k_2(n)+1}^{\infty} \sum_{\mathbf b \in \Sigma_B^k} \frac{p_{\mathbf b} }{2^{\max(n-1-j-k,1)\ell_{\mathbf b}^{-1} r_{\max}^{-1}}} \leq \sum_{k=k_2(n)+1}^{\infty} p_B^k = \frac{p_B^{k_2(n)+1}}{1-p_B} \leq \frac{p_B^{K_2 \cdot n}}{1-p_B}. \label{eq69b}
\end{align}
Combining \eqref{eq68c}, \eqref{eq68d} and \eqref{eq69b} yields 
\[\begin{split}
\sum_{k=0}^{\infty} \sum_{\mathbf b \in \Sigma_B^k}  \frac{p_{\mathbf b}}{2^{\max(n-1-j-k,1)\ell_{\mathbf b}^{-1} r_{\max}^{-1}}} \le & \frac{2^{-K_1 \cdot n^{1-a}}+ p_B^{K_2 \cdot n}}{1-p_B} + \frac{r_{\max}}{(1-\theta)K_2 \log 2} n^{\gamma-1} .
\end{split}\]
Since the first term on the right hand side decreases superpolynomially fast in $n$, this yields the existence of a constant $C_1 > 0$ and integer $n_1 \geq N_1$ for which the statement of the lemma holds.
\end{proof}

\begin{lemma}\label{lemma3.6c}
There exist $C_2 > 0$ and $n_2 \in \mathbb{N}$ such that for each integer $n \geq N_2$ we have
\begin{align*}
\sum_{k=0}^{\infty} \sum_{\mathbf b \in \Sigma_B^k}  \frac{p_{\mathbf b}}{2^{\max(n-1-k,1)\ell_{\mathbf b}^{-1} r_{\min}^{-1}}} \geq C_2 \cdot n^{\gamma_2-1}.
\end{align*}
\end{lemma}
\begin{proof}
Let $b \in \Sigma_B$ be such that $\gamma_2 = 1 + \frac{\log \pi_b}{\log \ell_b}$ with $\pi_b$ as in \eqref{eqpib}. For each $k \in \mathbb{N}$ let
\begin{align*}
A_k = \{\mathbf b = b_1 \cdots b_k \in \Sigma_B^k : \ell_{b_j} \geq \ell_b \text{ for each $j = 1,\ldots,k$ }\}.
\end{align*}
Then $\sum_{\mathbf b \in A_k} p_{\mathbf b} = \pi_b^k$ and for each $\mathbf b \in A_k$ we have $\ell_{\mathbf b} \geq \ell_b^k$. This gives
\begin{equation}\begin{split}
\sum_{k=0}^{\infty} \sum_{\mathbf b \in \Sigma_B^k}  \frac{p_{\mathbf b}}{2^{\max(n-1-k,1)\ell_{\mathbf b}^{-1} r_{\min}^{-1}}} & \geq \sum_{k=0}^{\infty} \sum_{\mathbf b \in A_k} \frac{p_{\mathbf b}}{2^{\max(n-1-k,1)\ell_{\mathbf b}^{-1} r_{\min}^{-1}}}\\
& \geq \sum_{k=0}^{\infty}  \frac{\pi_b^k}{2^{\max(n-1-k,1)\ell_b^{-k} r_{\min}^{-1}}}. \label{eq75b}
\end{split}\end{equation}
For each $n \in \mathbb{N}$ we define $k_3(n) = \ceil{\frac{(\gamma_2-1) \log n}{\log \pi_b}} = \ceil{\frac{\log n}{\log \ell_b}}$. Then $\pi_b^{k_3(n)-1} \geq n^{\gamma_2-1}$ and $\ell_b^{-k_3(n)} \leq n^{-1}$. We take $N_2 \in \mathbb{N}$ and $K_3 > 0$ such that for each integer $n \geq N_2$ we have $n-1-k_3(n) \geq 1$ and $(n-1-k_3(n)) \cdot n^{-1} \cdot r_{\min}^{-1} \leq K_3$. Then we get
\begin{equation}\begin{split}
\sum_{k=0}^{\infty} \frac{\pi_b^k}{2^{\max(n-1-k,1)\ell_b^{-k} r_{\min}^{-1}}} & \geq \pi_b^{k_3(n)} \frac{1}{2^{(n-1-k_3(n))\ell_b^{-k_3(n)} r_{\min}^{-1}}} \\
& \geq \pi_b \cdot n^{\gamma_2-1} \frac{1}{2^{(n-1-k_3(n)) \cdot n^{-1} \cdot r_{\min}^{-1}}} \\
& \geq \pi_b \cdot 2^{-K_3} \cdot n^{\gamma_2-1} \label{76b}
\end{split}\end{equation}
for each $n \ge N_2$. Combining \eqref{eq75b} and \eqref{76b} now yields the result with $C_2 = \pi_b \cdot 2^{-K_3}$.
\end{proof}

\begin{proof}[Proof of Proposition \ref{prop3.3b}]
First of all, note that
\begin{align*}
\int_Y \varphi \, d\mathbb{P} \times \lambda = \sum_{n=2}^{\infty} n \cdot \mathbb{P} \times \lambda(\varphi = n) \leq 2 \cdot \sum_{n=1}^{\infty} \mathbb{P} \times \lambda(\varphi > n).
\end{align*}
Since for each $\gamma < 0$ we have $\sum_{n=1}^{\infty} n^{\gamma-1} < \infty$, (i) follows from (ii). For (ii), let $\gamma \in (\gamma_1,0)$. It follows from Lemma \ref{lemma3.4c} and Lemma \ref{lemma3.5c} that for each integer $n \geq n_1$ we have
\[ \begin{split}
\mathbb{P} \times \lambda( \varphi > n ) \le \ & \frac{C_1}{4} \cdot n^{\gamma-1} \cdot \sum_{j=0}^{j(n)} s^j + \frac14 \sum_{j=j(n)+1}^{\infty} s^j \sum_{k=0}^{\infty} p_B^k \\
\leq \ & \frac{C_1}{4(1-s)} \cdot n^{\gamma-1} + \frac{1}{4(1-s)(1-p_B)} \cdot s^{j(n)+1}.
\end{split}\]
By the definition of $j(n)$, we have $s^{j(n)+1} \leq n^{\gamma-1}$, which gives (ii).

Finally, it follows from Lemma \ref{lemma3.4c} and Lemma \ref{lemma3.6c} that for each integer $n \geq n_2$ we have
\[ \mathbb{P} \times \lambda( \varphi > n ) \geq \frac{\min \{ p_g \, : \, g \in \Sigma_G \} \cdot C_2}{4}  \cdot n^{\gamma_2-1}.\qedhere \]
\end{proof}

\subsection{Estimates on the induced map}\label{subsec3.3}

Recall that $F^{\varphi}(\omega,x) = (\sigma^{\varphi(\omega,x)}\omega,T_{\omega}^{\varphi(\omega,x)}(x))$. The second part of the next lemma shows in particular that $F^{\varphi}$ projected on the second coordinate is expanding.

\begin{lemma}\label{lemma3.4a}
Let $(\omega,x) \in Y$.
\begin{itemize}
\item[(i)] For each $j=1,\ldots,\varphi(\omega,x)-1$ we have $DT_{\sigma^j \omega}^{\varphi(\omega,x)-j}(T_{\omega}^j(x)) \geq \frac{1}{2}$.
\item[(ii)] $DT_{\omega}^{\varphi(\omega,x)}(x) \geq 2$.
\end{itemize}
\end{lemma}

\begin{proof}
Let $(\omega,x) \in Y$. Recall the definitions of $\kappa = \kappa(\omega,x)$ from \eqref{q:to12}, $l = l(\omega,x)$ from \eqref{q:l} and $m= m(\omega,x)$ from \eqref{eq32a}. Write\footnote{We use different letters here than for the partition elements $P_{u \mathbf v \mathbf b g \mathbf w}$ from $\mathcal P$, since the subdivision here is different (and $(\omega,x)$-dependent).}
\[ \begin{split}
\mathbf u =\, & \omega_2 \cdots \omega_m \in \Sigma^*,\\
\mathbf d =\, & \omega_{m+1} \cdots \omega_{\kappa-1} \in \Sigma_B^*,\\
g =\, & \omega_{\kappa} \in \Sigma_G.
\end{split}\]
Then 
\[ T_{\omega}^{\varphi(\omega,x)}(x) = R^l \circ L_g \circ L_{\mathbf d} \circ L_{\mathbf u} \circ R(x).\]
We have $DL_b(y) \geq 1$ for all $y \in [0,x_b)$ and all $b \in \Sigma_B$. For $v \in \Sigma_G$ with $r_v > 1$ we obtained in Lemma \ref{lemma2.6d}(ii) that $x_v > L_v^{-1} (\frac{1}{2})$ and hence $DL_v(y) \geq 1$ for all $v \in \Sigma_G$ and $y \in [0,L_v^{-1} (\frac{1}{2}))$. It follows from the definition of $m$ that, for each $j \in \{1,\ldots,m-1\}$,
\begin{align}\label{eq82b}
DL_{u_j \cdots u_{m-1}}(L_{u_1 \cdots u_{j-1}}\circ R(x)) = \prod_{i=j}^{m-1} DL_{u_i}(L_{u_1 \cdots u_{i-1}} \circ R(x)) \geq 1.
\end{align}
Let $\mathbf q,\mathbf t \in \Sigma_B^*$ be any two words such that $\mathbf d = \mathbf q \mathbf t$.
Using Lemma \ref{lemma2.1} we find that for each $y \in [0,\frac{1}{2})$,
\begin{align}\label{eq83b}
D(L_g \circ L_{\mathbf t})(y) = 2\ell_{\mathbf t} r_g (1-2y)^{\ell_{\mathbf t}r_g-1}.
\end{align}
Furthermore, from \eqref{q:2tol} we see that
\begin{align}\label{eq84b}
2^l \geq \frac14 (1-2L_{\mathbf u}\circ R(x))^{-\ell_{\mathbf d}r_g}
\end{align}
and applying Lemma~\ref{lemma2.1} to $L_\mathbf q$ gives
\begin{equation}\label{q:ells}
(1-2L_{\mathbf q} \circ L_{\mathbf u} \circ R(x))^{\ell_{\mathbf t}r_g-1} = (1-2L_{\mathbf u}\circ R(x))^{\ell_{\mathbf d}r_g - \ell_{\mathbf q}}.
\end{equation}
Combining \eqref{eq83b}, \eqref{eq84b} and \eqref{q:ells} yields
\begin{equation}\begin{split}
D(R^l \circ L_g \circ L_{\mathbf t}) & (L_{\mathbf q} \circ L_{\mathbf u} \circ R(x)) = 2^l \cdot 2 \ell_{\mathbf t} r_g (1-2L_{\mathbf q} \circ  L_{\mathbf u}\circ R(x))^{\ell_{\mathbf t} r_g-1} \\
& \geq \frac{1}{4} (1-2L_{\mathbf u}\circ R(x))^{-\ell_{\mathbf d}r_g} \cdot 2\ell_{\mathbf t} r_g (1-2L_{\mathbf u}\circ R(x))^{\ell_{\mathbf d}r_g-\ell_{\mathbf q}} \label{eq88b} \\
& = \frac{1}{2} \ell_{\mathbf t} r_g (1-2L_{\mathbf u} \circ R(x))^{-\ell_{\mathbf q}},
\end{split}\end{equation}
which we can lower bound by $\frac{1}{2}$. To prove (i), for any $j \in \{1,\ldots, m-1\}$ taking $\mathbf q = \epsilon$ (which means $\ell_{\mathbf q}=1$) and $\mathbf t = \mathbf d$ we obtain using \eqref{eq82b} that 
\begin{align*}
DT_{\sigma^j \omega}^{\varphi(\omega,x)-j}(T_{\omega}^jx) = D(R^l\circ L_g\circ L_{\mathbf d})(L_{\mathbf u} \circ R(x)) \cdot DL_{u_j \cdots u_{m-1}}(L_{u_1 \cdots u_{j-1}}\circ R(x)) \geq \frac{1}{2}.
\end{align*}
For $j \in \{m,\ldots,\kappa-1\}$ we take $\mathbf q = \omega_{m+1} \cdots \omega_j$ and $\mathbf t = \omega_{j+1}\cdots \omega_{\kappa-1}$ (with $\mathbf q = \epsilon$ in case $j=m$ and $\mathbf t = \epsilon$ in case $j=\kappa-1$) and get
\[ DT_{\sigma^j \omega}^{\varphi(\omega,x)-j}(T_{\omega}^j(x)) = D(R^l\circ L_g \circ L_{\mathbf t})(L_{\mathbf q}\circ L_{\mathbf u} \circ R(x)) \geq \frac{1}{2}\]
by \eqref{eq88b}. Finally, if $j \in \{\kappa,\ldots, \varphi(\omega,x)-1\}$, then
\[ DT_{\sigma^j \omega}^{\varphi(\omega,x)-j}(T_{\omega}^j(x))  = 2^{ \kappa+l-j} \geq \frac{1}{2}.\]
This proves (i). For (ii), we write
\[ DT_{\omega}^{\varphi(\omega,x)} (x) = D(R^l\circ L_g \circ L_{\mathbf d})(L_{\mathbf u} \circ R(x)) \cdot DL_{\mathbf u}(R(x)) \cdot DR(x).\]
We have $DR(x)=2$ and by \eqref{eq82b} with $j=1$ we get $DL_{\mathbf u}(R(x)) \ge 1$. What is left is to estimate the first factor. From \eqref{eq88b} with $\mathbf q = \epsilon$ and $\mathbf t = \mathbf d$ we see that
\[ D(R^l\circ L_g \circ L_{\mathbf d})(L_{\mathbf u} \circ R(x)) \ge \frac12 \ell_{\mathbf d} r_g (1-2 L_{\mathbf u} \circ R(x))^{-1}.\]
Note that by the definition of $m$ we have $L_{\mathbf u} \circ R(x) \in (L_g^{-1}(\frac{1}{2}),\frac{1}{2})$ if $m =\kappa-1$, so if $\mathbf d = \epsilon$, and $L_{\mathbf u} \circ R(x) \in [x_{d_1},\frac{1}{2})$ if $m <  \kappa- 1$. In case $m = \kappa-1$ we obtain that
\[ \ell_{\mathbf d} r_g (1-2L_{\mathbf u} \circ R(x))^{-1} \geq r_g \Big(1-2L_g^{-1}\Big(\frac{1}{2}\Big)\Big)^{-1} = r_g \cdot 2^{1/r_g} \geq 2,\]
where we used the expression for $L_g^{-1}(\frac12)$ from \eqref{eq33f} and the fact that $x \cdot 2^{1/x} \geq 2$ for all $x \geq 1$. In case $m < \kappa-1$, we have
\[ \ell_{\mathbf d} r_g (1-2L_{\mathbf u} \circ R(x))^{-1} \geq \ell_{d_1} (1-2x_{d_1})^{-1} =  \ell_{d_1}^{1+(\ell_{d_1}-1)^{-1}} \geq 2, \]
where we used \eqref{eq36f} and the fact that $x^{1+(x-1)^{-1}} > 2$ for all $x > 1$. Hence, in all cases
\[ DT_{\omega}^{\varphi(\omega,x)} (x) \geq \frac{1}{2} \cdot 2 \cdot 1 \cdot 2 = 2. \qedhere\]
\end{proof}

Recall the first return time partition $\mathcal P$ from \eqref{q:frtp}. For $P = P_{u \mathbf v \mathbf b g \mathbf w} \in \mathcal P$ set
\[ \pi_2 (P) := (R^{|\mathbf w|} \circ L_g \circ L_{\mathbf b} \circ L_{\mathbf v} \circ R)^{-1} \Big( \frac12, \frac34 \Big)\]
and write $S_P$ for the restriction of the map $T_\omega^{\varphi(\omega,x)}$ to $\pi_2(P)$, so
\[ S_P := T_\omega^{\varphi(\omega,x)}|_{\pi_2(P)} = R^{|\mathbf w|} \circ L_g \circ L_{\mathbf b} \circ L_{\mathbf v} \circ R|_{\pi_2(P)}.\]
We give two lemma's on the maps $S_P$, that will be useful when verifying (t6) for the Young tower in the next section.

\begin{lemma}\label{lemma3.8}
There exists a constant $C_3 > 0$ such that for each $P \in \mathcal P$ and all $(\omega,x), (\omega',y) \in P$ we have
\begin{align*}
\Big|\frac{J_{\mathbb{P} \times \lambda} F^{\varphi}(\omega,x)}{J_{\mathbb{P} \times \lambda} F^{\varphi}(\omega',y)}-1\Big| \leq C_3 \cdot \big| S_P(x) - S_P(y) \big|.
\end{align*}
\end{lemma}

\begin{proof}
For each $P \in \mathcal P$ and all $(\omega,x), (\omega',y) \in P$ we have $\varphi(\omega,x)= \varphi(\omega',y) = \varphi_P$ and $\omega_j = \omega'_j$ for all $1 \le j \le \varphi_P$. Hence, for each measurable set $A \subseteq P$ we have
\[ \mathbb{P} \times \lambda(F^{\varphi}(A)) = \int_A \Big( \prod_{j=1}^{\varphi(\omega,x)} p^{-1}_{\omega_j} \Big) DT_{\omega}^{\varphi(\omega,x)}(x) \, d\mathbb{P} \times \lambda(\omega,x).\]
By Proposition~\ref{prop2.1a} we obtain
\[ J_{\mathbb{P} \times \lambda} F^{\varphi}(\omega,x) = \Big( \prod_{j=1}^{\varphi(\omega,x)} p^{-1}_{\omega_j} \Big) DT_{\omega}^{\varphi(\omega,x)}(x),\]
which, for each $P \in \mathcal P$ and all $(\omega,x), (\omega',y) \in P$, gives
\[ \Big|\frac{J_{\mathbb{P} \times \lambda} F^{\varphi}(\omega,x)}{J_{\mathbb{P} \times \lambda} F^{\varphi}(\omega',y)}-1\Big| = \Big| \frac{DT_{\omega}^{\varphi(\omega,x)}(x)}{DT_{\omega'}^{\varphi(\omega',y)}(y)} -1 \Big| = \Big| \frac{DS_P(x)}{DS_P(y)} -1 \Big|.\]
Let $c > 0$. As compositions of good and bad maps each $S_P$ has non-positive Schwarzian derivative and by Lemma \ref{lemma3.4a}(ii) each $S_P$ satisfies $DS_P \geq 2$. For this reason for each $P \in \mathcal{P}$ we can extend the domain $\pi_2(P)$ of $S_P$ on both sides to an interval $I_P \supseteq \pi_2(P)$ such that there exists an extension $\tilde{S}_P: I_P \rightarrow \mathbb{R}$ of $S_P$, i.e.~$\tilde{S}_P|_{\pi_2(P)} = S_P|_{\pi_2(P)}$, that has non-positive Schwarzian derivative and for which both components of $\tilde{S}_P(I_P) \backslash (\frac{1}{2},\frac{3}{4})$ have length at least $\frac{c}{4}$. Applying for each $P \in \mathcal{P}$ the Koebe Principle \eqref{eq2.5} with $I = I_P$ and $J = \pi_2(P)$ then gives a constant $C_3 > 0$ that depends only on $c$ such that for each $P \in \mathcal{P}$ and each  $x,y \in \pi_2(P)$ we have
\[ \Big| \frac{DS_P(x)}{DS_P(y)} -1 \Big| \leq C_3 \cdot |S_P(x) - S_P(y)|. \]
This gives the lemma.
\end{proof}

Recall the definition of the separation time from \eqref{q:septimetower}. We have the following lemma.

\begin{lemma}
Let $(\omega,x),(\omega',y) \in Y$. Then
\begin{align}\label{eq113d}
|x-y| \leq 2^{-s((\omega,x),(\omega',y))}.
\end{align}
Furthermore, if $(\omega,x),(\omega',y) \in P$ for some $P \in \mathcal P$, then
\begin{align}\label{eq114d}
\big|S_P (x)-S_P (y) \big| \leq 2^{-s(F^\varphi (\omega,x) ,F^\varphi (\omega',y))}.
\end{align}
\end{lemma}

\begin{proof}
Write $n=s((\omega,x),(\omega',y))$ and, for each $k \in \{0,1,\ldots,n-1\}$, let $P^{(k)} \in \mathcal P$ be such that $(F^{\varphi})^{k}(\omega,x), (F^{\varphi})^{k}(\omega',y)  \in P^{(k)}$. Then for each $k \in \{0,1,\ldots,n-1\}$ the points $(S_{P^{(k-1)}} \circ \cdots \circ S_{P^{(0)}})(x)$ and $(S_{P^{(k-1)}} \circ \cdots \circ S_{P^{(0)}})(y)$ lie in the domain $\pi_2(P^{(k)})$ of $S_{P^{(k)}}$, so it follows from Lemma \ref{lemma3.4a}(ii) together with the Mean Value Theorem that
\[ \frac{|S_{P^{(k)}} \circ  \cdots \circ S_{P^{(0)}}(x)-S_{P^{(k)}} \circ \cdots \circ S_{P^{(0)}}(y)|}{|S_{P^{(k-1)}} \circ \cdots \circ S_{P^{(0)}}(x)-S_{P^{(k-1)}} \circ \cdots \circ S_{P^{(0)}}(y)|} \geq \inf DS_{P^{(k)}} \geq 2.\]
We conclude that
\[ \begin{split}
|x-y| \leq \ & 2^{-1} |S_{P^{(0)}} (x) - S_{P^{(0)}} (y)|\\
 \leq \ & \cdots \leq 2^{-n} |S_{P^{(n-1)}} \circ \cdots \circ S_{P^{(0)}} (x) - S_{P^{(n-1)}} \circ \cdots \circ S_{P^{(0)}} (y)| \leq 2^{-n},
\end{split}\]
which gives the first part of the lemma. For the second part, note that if $(\omega,x),(\omega',y) \in P$ for some $P\in \mathcal P$, then \eqref{eq114d} follows by applying \eqref{eq113d} to the points $F^{\varphi}(\omega,x) = (\sigma^{\varphi_P} \omega,S_P(x))$ and $F^{\varphi}(\omega',y) = (\sigma^{\varphi_P} \omega',S_P(y))$.
\end{proof}

\section{A Young tower for the random map} \label{sec4}

\subsection{The acs probability measure}\label{subsec4.1}
We are now in the position to construct a Young tower for the skew product $F$ according to the set-up from \cite[Section 1.1]{Y99} that we outlined in Section~\ref{s:youngtower}.

\vskip .2cm
As the base for the Young tower we take the set $Y$. The Young tower $\Delta$, the $l^{\text{th}}$ levels of the tower $\Delta_l$ and the tower map $G: \Delta \rightarrow \Delta$ are defined in Section~\ref{s:youngtower}, as well as the reference measure $m$ and the partition $\eta$ on $\Delta$. Following the general setup in \cite{Y99}, inducing the map $G$ on $\Delta_0 = Y \times \{0\}$ yields a transformation $G^{\varphi}$ on $\Delta_0$ given by $G^{\varphi}(z,0) = G^{\varphi(z)}(z,0)$. Recall that we identify $G^{\varphi}$ with $F^{\varphi}$ by identifying $\Delta_0$ with $Y$ and using the correspondence $G^{\varphi}(z,0) = (F^{\varphi}(z),0)$. We check that the conditions (t1)--(t6) from Section~\ref{s:youngtower} hold for this construction.

\begin{prop}\label{prop4.1n}
The conditions (t1)--(t6) hold for the map $G$ on the Young tower $\Delta$ defined above.
\end{prop}

\begin{proof}
From the first return time partition $\mathcal P$ it is clear that (t1), (t2), (t3) and (t5) hold.

For (t4) it is enough to show that the collection
\[ \bigvee_{n \ge 0} G^{-n} \eta = \{ E_0 \cap G^{-1} E_1 \cap \cdots \cap G^{-n} E_n \, : \, E_i \in \eta,\, 1 \le i \le n, \, n \ge 0 \}\]
separates points. To show this, let $(z_1, l_1), (z_2, l_2) \in \Delta$ be two points. If $z_1 = z_2$ and $l_1 \neq l_2$ and $P \in \mathcal P$ is such that $z_1 \in P$, then $\Delta_{l_1,P}, \Delta_{l_2,P} \in \eta$ are sets that separate $(z_1, l_1)$ and $(z_2, l_2)$. Assume that $z_1 \neq z_2$. Lemma~\ref{lemma3.4a}(ii) implies that the map $F^\varphi$ is expanding on $Y$, so there exist an $N \ge 0$ and two disjoint sets $A, E \in \bigvee_{n=0}^N (F^\varphi)^{-n} \mathcal P$ such that $z_1 \in A$ and $z_2 \in E$. On these sets the first $N$ first return times to $Y$ are constant, meaning that if $K>0$ is such that $G^K(z_1,0) = ((F^\varphi)^N(z_1),0)$, then $A \times \{ l_1 \} \in \bigvee_{n=0}^{K+l_1} G^{-n}\eta$ and if $L>0$ such that $G^L(z_2,0) = ((F^\varphi)^N(z_2),0)$, then $E \times \{ l_2 \} \in \bigvee_{n=0}^{L+l_2} G^{-n}\eta$. Note that $(z_1, l_1) \in A \times \{l_1\}$ and $(z_2, l_2) \in E \times \{l_2\}$ and $A \times \{ l_1 \} \cap E \times \{l_2\} = \emptyset$. Hence, (t4) holds.

Finally, from Lemma \ref{lemma3.8} and \eqref{eq114d} we obtain that
\begin{align}\label{eq92a}
\Big|\frac{J_{\mathbb{P} \times \lambda} F^{\varphi}(z_1)}{J_{\mathbb{P} \times \lambda} F^{\varphi}(z_2)}-1\Big| \leq C_3 \cdot 2^{-s(F^\varphi (z_1) ,F^\varphi (z_2))}
\end{align}
for each $P \in \mathcal P$ and all $z_1,z_2 \in P$. This gives (t6) with $\beta = \frac12$ and the proposition follows.
\end{proof}

Now Proposition~\ref{prop3.3b}(i) and Theorem~\ref{t:young1} imply the existence of a probability measure $\nu$ on $(\Delta, \mathcal B)$ that is $G$-invariant, exact and absolutely continuous with respect to $m$ with a density that is bounded and bounded away from zero and that satisfies \eqref{q:densityC}. We use this to construct the invariant measure for $F$ that is promised in Theorem~\ref{result1b}. Define
\[ \pi : \Delta \rightarrow \Sigma^{\mathbb{N}} \times [0,1], \, (z,l) \mapsto F^l(z).\]
Then
\[ \begin{split}
 \pi(G(z,l)) =\ & \pi(z,l+1) = F^{l+1}(z) = F(\pi(z,l)), \qquad  l < \varphi(z)-1, \\
 \pi(G(z,l)) =\ & \pi(F^{\varphi}(z),0) = F^{\varphi}(z) = F(\pi(z,l)), \qquad l = \varphi(z)-1.
\end{split}\]
So $\pi \circ G = F \circ \pi$. Let $\rho = \nu \circ \pi^{-1}$ be the pushforward measure of $\nu$ under $\pi$.

\begin{lemma}\label{lemma4.1e}
The probability measure $\rho$ satisfies the following properties.
\begin{itemize}
\item[(i)] $F$ is measure preserving and mixing with respect to $\rho$.
\item[(ii)] $\rho$ is absolutely continuous with respect to $\mathbb{P} \times \lambda$.
\item[(iii)] We have
\[ \rho(A \cap Y) = \nu((A \cap Y) \times \{0\}), \qquad A \in \mathcal F.\]
\end{itemize}
\end{lemma}

\begin{proof}
Part (i) immediately follows from the properties of the measure $\nu$ and the fact that $\pi \circ G = F \circ \pi$. For (ii), let $A \in \mathcal F$ be such that $\mathbb{P} \times \lambda(A) = 0$. Using that $F$ is non-singular with respect to $\mathbb{P} \times \lambda$, we obtain that
\begin{align*}
m(\pi^{-1}(A)) = m\Big(\Delta \cap \Big(\bigcup_{l \ge 0} F^{-l} (A) \times \{l\}\Big)\Big) \leq \sum_{l \ge 0} \mathbb{P} \times \lambda(F^{-l}(A)) = 0.
\end{align*}
Since $\nu$ is absolutely continuous with respect to $m$, it follows that $\rho(A) = \nu(\pi^{-1}A) = 0$. For (iii) let $A \in \mathcal F$. We have
\[ \pi^{-1}(A \cap Y) = \bigcup_{P \in \mathcal P} \bigcup_{l=0}^{\varphi_P-1} (F^{-l} (A \cap Y) \cap P) \times \{l\}.\]
By definition of $\varphi_P$, we have $F^l(z) \notin Y$ for each $z \in P$ and each $l \in \{1,\ldots,\varphi_P-1\}$. Therefore
\[ \pi^{-1}(A \cap Y) = \bigcup_{P \in \mathcal P} (A \cap P) \times \{0\} = (A \cap Y) \times \{0\}. \qedhere\]
\end{proof}

Combining Lemma \ref{lemma4.1e} with Lemma \ref{l:productmeasure} yields that there exists a probability measure $\mu$ that is absolutely continuous with respect to $\lambda$ and such that $\rho = \mathbb{P} \times \mu$. In other words, $\mu$ is an acs measure for $F$. We will now prove Theorem \ref{result1b}, which shows that $\mu$ is in fact the only acs measure for $F$.

\begin{proof}[Proof of Theorem \ref{result1b}]
It follows from Lemma \ref{lemma4.1e}(i) that $F$ is mixing with respect to $\mathbb{P} \times \mu$. Hence, to obtain that $\mu$ is the only acs measure for $F$, it suffices to show that $\frac{d\mu}{d\lambda} > 0$ holds $\lambda$-a.e. Theorem~\ref{t:young1} asserts that there is a constant $C_4 \ge 1$ such that
\begin{align}\label{eq119y}
\frac{1}{C_4} \leq \frac{d\nu}{dm} \leq C_4.
\end{align}
Let $B  \subseteq (\frac{1}{2},\frac{3}{4})$ be a Borel set. Lemma \ref{lemma4.1e}(iii) and \eqref{eq119y} imply that
\begin{equation*}
\mu(B) \ge \nu((\tilde{\Omega} \times B) \cap Y\times \{0\}) \geq C_4^{-1} \cdot m((\tilde{\Omega} \times B) \cap Y \times \{0\}) = C_4^{-1} \cdot \lambda(B).
\end{equation*}
Since $B$ was arbitrary, we have $\frac{d\mu}{d\lambda}(x) \geq C_4^{-1}$ for $\lambda$-a.e.~$x \in (\frac12, \frac34)$. Recall that the density $\frac{d\mu}{d\lambda}$ is a fixed point of the Perron-Frobenius operator $\mathcal{P}_{F,\mathbf p}$ from \eqref{eqn3.22}. Fix some $g \in \Sigma_G$. The map $T_g^2|_{(\frac{1}{2},\frac{3}{4})}: (\frac{1}{2},\frac{3}{4}) \rightarrow (0,1)$ is a measurable bijection with measurable inverse. For each $x \in (0,1)$ let $y_x$ be the unique element in $(\frac{1}{2},\frac{3}{4})$ that satisfies $x = T_g^2(y_x)$. Furthermore, note that $\sup_{y \in (\frac12, \frac34)} DT_g^2(y) \le 2DT_g(0)$. We conclude that for $\lambda$-a.e.~$x \in (0,1)$
\[ \frac{d\mu}{d\lambda}(x) = \mathcal{P}_{F,\mathbf p}^2 \frac{d\mu}{d\lambda} (x) \geq p_g^2 \frac{\frac{d\mu}{d\lambda} (y_x)}{DT_g^2(y_x)} \geq  \frac{p_g^2C_4^{-1}}{2DT_g(0)} >0. \]
This gives that $\mu$ is the unique acs measure and that the density $\frac{d\mu}{d\lambda}$ is bounded away from zero.
\end{proof}

\begin{remark}\label{r:lipschitz}{\rm
Besides Theorems \ref{result1a} and \ref{result1b} it can also be shown that all the results from Theorem 1.2 in \cite{Zeegers21} carry over. Namely, by following the same steps as in Section 3 of \cite{Zeegers21} it can be shown that $F$ admits, independent of the value of $\theta$, a unique (up to scalar multiplication) acs measure that is $\sigma$-finite and ergodic and for which the density is bounded away from zero, is locally Lipschitz on $(0,\frac{1}{2})$ and $[\frac{1}{2},1)$ and is not in $L^q$ for any $q > 1$. This measure is infinite if $\theta \geq 1$ and coincides with $\mu$ if $\theta < 1$. 
}\end{remark}

\subsection{Decay of correlations}

Recall that for $\alpha \in (0,1)$ we have set $\mathcal H_\alpha$ for the set of $\alpha$-H\"older continuous functions on $\Sigma^\mathbb N \times [0,1]$ with metric $d$ as in \eqref{q:metric}. Also recall the definition of the function spaces $\mathcal C_\delta$ on $\Delta$ from \eqref{q:cbeta}.

\begin{lemma}\label{lemma4.2d}
Let $\alpha \in (0,1)$ and $h \in \mathcal{H}_{\alpha}$. Then $h \circ \pi \in \mathcal{C}_{1/2^{\alpha}}$.
\end{lemma}

\begin{proof}
Since $h \in \mathcal{H}_{\alpha}$, there exists a constant $C_5 > 0$ such that 
\begin{align}\label{eq127d}
|h(z_1)-h(z_2)| \leq C_5 \cdot d(z_1,z_2)^{\alpha} \quad \text{ for all } z_1,z_2 \in \Sigma^{\mathbb{N}} \times [0,1].
\end{align}
From this it is easy to see that $\|h\|_{\infty} < \infty$. Let $\upsilon_1 = (z_1, l_1), \upsilon_2 = (z_2, l_2) \in \Delta$. If $l_1\neq l_2$ or if $z_1$ and $z_2$ lie in different elements of $\mathcal P$, then $s(\upsilon_1, \upsilon_2) =0$ and
\begin{equation}\label{q:gnapi}
|h\circ \pi(\upsilon_1)-h \circ \pi (\upsilon_2)| = |h(F^{l_1}(z_1))- h(F^{l_2}(z_2))| \le \| h \|_\infty = \| h \|_\infty \cdot 2^{-\alpha s(\upsilon_1, \upsilon_2)}.
\end{equation}
Hence, to prove that $h \circ \pi \in \mathcal{C}_{1/2^\alpha}$, it remains to consider the case that $z_1,z_2 \in P$ for some $P \in \mathcal P$ and $l_1 = l_2 = l \in \{0, \ldots, \varphi_P-1\}$. Write $z_1 = (\omega,x)$ and $z_2 =(\omega',y)$. Note that $\omega_j = \omega_j'$ for each $j \in \{1,2,\ldots,\varphi_P\}$. Hence,
\[ 2^{-\min \{ i \in \mathbb N \, : \, \omega_{l + i} \neq \omega_{l +i}' \}} \le 2^{-\min \{ i \in \mathbb N \, : \, \omega_{\varphi_P-1 + i} \neq \omega_{\varphi_P-1 +i}' \}} \le 2^{-s(z_1, z_2)}.\]
Furthermore, it follows from the Mean Value Theorem together with Lemma \ref{lemma3.4a}(i) that
\[ \frac{|T_{\omega}^{\varphi_P}(x) - T_{\omega}^{\varphi_P}(y)|}{|T_{\omega}^l(x) - T_{\omega}^l(y)|} = \frac{|T_{\sigma^l \omega}^{\varphi_P-l}(T_{\omega}^l (x)) - T_{\sigma^l \omega}^{\varphi_P-l}(T_{\omega}^l (y))|}{|T_{\omega}^l(x) - T_{\omega}^l(y)|} \geq \frac{1}{2}.\]
Combining this with \eqref{eq114d} yields that
\[ |T_{\omega}^l(x) - T_{\omega}^l(y)| \leq 2 \cdot 2^{-s(F^{\varphi}(z_1),F^{\varphi}(z_2))} = 4 \cdot 2^{-s(z_1,z_2)}\]
and hence by \eqref{eq127d},
\[ |h(F^{l}(z_1))- h(F^{l}(z_2))| \le C_5 (2^{-s(z_1,z_2)} + 4 \cdot 2^{-s(z_1,z_2)})^\alpha = 5^\alpha C_5 \cdot 2^{-\alpha s(z_1,z_2)}. \]
Together with \eqref{q:gnapi} this gives the result.
\end{proof}

We now have all the ingredients to prove Theorem \ref{result2a}.

\begin{proof}[Proof of Theorem \ref{result2a}]
To prove the theorem, we would like to use Theorem~\ref{t:young1}(iv), which requires us to bound $m(\hat \varphi >n)$, where
\[ \hat{\varphi} : \Delta \rightarrow \mathbb{N}_0, \,  \upsilon \mapsto \inf \{n \ge 0 \, :\, G^n (\upsilon) \in \Delta_0 \}.\]
Since
\[ \begin{split}
\{\hat{\varphi} = 0\} =\ & \Delta_0 = \bigcup_{P \in \mathcal P: \varphi_P > 0} \Delta_{0,P}, \\
\{\hat{\varphi} = n\} =\ & \bigcup_{P \in \mathcal P : \varphi_P > n} \Delta_{\varphi_P-n,P}, \qquad n \ge 1,
\end{split}\]
we have for each $n \ge 0$ that
\[ m(\hat{\varphi} = n) =\sum_{P \in \mathcal P : \varphi_P > n} \mathbb{P} \times \lambda(P) = \mathbb{P} \times \lambda(\varphi > n).\]
It follows from Proposition \ref{prop3.3b}(ii) that for each $\gamma \in (\gamma_1,0)$ there is an $M >0$ and an $N \ge 1$ such that for each $n \ge N$,
\[ \mathbb{P} \times \lambda(\varphi > n) \le M \cdot n^{\gamma-1}.\]
Thus, for all $n \geq N$,
\begin{align}\label{eq140g}
m(\hat{\varphi} > n) = \sum_{k > n} m(\hat{\varphi} = k) \le M \sum_{k \geq n} k^{\gamma-1} \le M \cdot n^{\gamma-1} + 
M \int_n^{\infty} x^{\gamma-1} dx. 
\end{align}
So, $m(\hat{\varphi} > n) = O(n^\gamma)$. Combining Proposition \ref{prop3.3b}(i), Proposition~\ref{prop4.1n} and Theorem~\ref{t:young1}(iv) now gives that for each $\gamma \in (\gamma_1,0)$, $\hat f \in L^{\infty}(\Delta,\nu)$, $\delta \in (0,1)$ and $\hat h \in \mathcal{C}_{\delta}$,
\begin{align}\label{eq138g}
\Big|\int_{\Delta} \hat f \circ G^n \cdot \hat h \, d\nu - \int_{\Delta} \hat f \, d\nu  \int_{\Delta} \hat h \, d\nu \Big| = O(n^{\gamma}).
\end{align}
Now, let $\gamma \in (\gamma_1,0)$, $f \in L^{\infty}(\Sigma^{\mathbb{N}} \times [0,1], \mathbb{P} \times \mu)$ and $h \in \mathcal{H}$. Using that $\mathbb{P} \times \mu = \nu \circ \pi^{-1}$ and $\pi \circ G = F \circ \pi$, it then follows that 
\[ |Cor_n(f,h)| = \Big|\int_{\Delta} (f \circ \pi) \circ G^n \cdot (h \circ \pi) \, d\nu - \int_{\Delta} f\circ \pi  \, d\nu  \int_{\Delta} h \circ \pi \, d\nu \Big|.\]
Since $h \in \mathcal H$, it holds that $h \in \mathcal{H}_{\alpha}$ for some $\alpha \in (0,1)$, so $h \circ \pi \in \mathcal{C}_{1/2^\alpha}$ by Lemma \ref{lemma4.2d}. Since also $f \circ \pi \in L^{ \infty}(\Delta,\nu)$, we obtain the result from \eqref{eq138g} with $\hat f = f \circ \pi$ and $\hat h = h \circ \pi$. 
\end{proof}



In order to prove Theorem \ref{result2b}, we need the following lemma.

\begin{lemma}\label{lemma4.3p}
There exists a constant $C_6 > 0$ such that for each $P \in \mathcal P$ and $z_1,z_2 \in P$,
\[ \Big| \log \frac{J_{\mathbb{P} \times \mu} F^{\varphi}(z_1)}{J_{\mathbb{P} \times \mu} F^{\varphi}(z_2)}\Big| \leq C_6 \cdot 2^{-s(z_1,z_2)}.\]
\end{lemma}

\begin{proof}
From the definition of the Jacobian we see that $J_m G^\varphi|_{\Delta_0} = J_{\mathbb P \times \lambda} F^\varphi$ with the identification of $\Delta_0$ and $Y$. Lemma \ref{lemma4.1e}(iii) and Lemma \ref{lemma2.1a} give us that for each $P \in \mathcal P$ and each measurable set $A \subseteq P$,
\[ \begin{split}
\mathbb{P} \times \mu(F^{\varphi}(A)) =\ & \nu(G^{\varphi}(A \times \{0\})) \\
=\ & \int_{A \times \{0\}} \Big(\frac{d\nu}{dm} \circ G^{\varphi}\Big) J_m G^{\varphi} \, dm \\
=\ & \int_A \frac{d\nu}{dm}(F^{\varphi}(z),0) \cdot J_{\mathbb{P} \times \lambda} F^{\varphi}(z) \cdot \frac{dm}{d\nu}(z,0) \, d\mathbb{P} \times \mu(z).
\end{split}\]
This gives
\[ J_{\mathbb{P} \times \mu} F^{\varphi}(z) = \frac{d\nu}{dm}(F^{\varphi}(z),0) \cdot J_{\mathbb{P} \times \lambda} F^{\varphi}(z) \cdot \frac{dm}{d\nu}(z,0), \qquad z \in Y,\]
and thus, for each $z_1,z_2 \in Y$,
\[ \Big| \log \frac{J_{\mathbb{P} \times \mu} F^{\varphi}(z_1)}{J_{\mathbb{P} \times \mu} F^{\varphi}(z_2)}\Big| \leq 
 \Big| \log \frac{\frac{d\nu}{dm}(F^{\varphi}(z_1),0)}{\frac{d\nu}{dm}(F^{\varphi}(z_2),0)}\Big|+
 \Big| \log \frac{J_{\mathbb{P} \times \lambda} F^{\varphi}(z_1)}{J_{\mathbb{P} \times \lambda} F^{\varphi}(z_2)}\Big| +
 \Big| \log \frac{\frac{d\nu}{dm}(z_2,0)}{\frac{d\nu}{dm}(z_1,0)}\Big|. \]
Combining Proposition \ref{prop3.3b}(i), Proposition~\ref{prop4.1n} and Theorem~\ref{t:young1}(ii) gives the existence of a constant $C^+ > 0$ such that, for each $\Delta_{l,P} \in \eta$ and $\upsilon_1,\upsilon_2 \in \Delta_{l,P}$,
\begin{align}\label{eq119g}
\Big|\frac{\frac{d\nu}{dm}(\upsilon_1)}{\frac{d\nu}{dm}(\upsilon_2)}-1\Big| \leq C^+ \cdot 2^{-s(\upsilon_1,\upsilon_2)}. 
\end{align}
Using that $|\log \frac{x}{y}| \in \{ \log \frac{x}{y}, \log \frac{y}{x} \}$ and $|\log x| \leq \max\{|x-1|,|x^{-1}-1|\}$ for all $x,y>0$, we obtain from \eqref{eq92a} and \eqref{eq119g} that
\[ \Big| \log \frac{J_{\mathbb{P} \times \mu} F^{\varphi}(z_1)}{J_{\mathbb{P} \times \mu} F^{\varphi}(z_2)}\Big| \le C^+ \cdot 2^{-s(F^\varphi(z_1), F^\varphi(z_2))} + C_3 \cdot 2^{-s(F^\varphi(z_1), F^\varphi(z_2))} + C^+ \cdot 2^{-s(z_1,z_2)}\]
for all $z_1, z_2 \in P$, $P \in \mathcal P$. The lemma thus holds with $C_6 = 3C^+ + 2C_3$.
\end{proof}

\begin{proof}[Proof of Theorem \ref{result2b}]
Let $f \in L^{\infty}(\Sigma^{\mathbb{N}} \times [0,1], \mathbb{P} \times \mu)$ and $h \in \mathcal{H}$ be such that both $f$ and $h$ are identically zero on $\Sigma^{\mathbb{N}} \times \big([0,\frac{1}{2}] \cup [\frac{3}{4},1]\big)$ and such that $\int f \, d\mathbb{P} \times \mu \cdot \int h \, d\mathbb{P} \times \mu > 0$. Let $\gamma \in (\gamma_1,\min\{\gamma_2+1,-1\})$ if $\gamma_1 < -1$ and $\gamma \in (\gamma_1,\frac{\gamma_2}{2})$ if $-1 \leq \gamma_1 < 0$. This is possible by assumption. Our strategy is to apply Theorem \ref{t:gouezel} with $Y$ as before. For this, we verify (g1), (g2) and (g3).\footnote{More precisely, we apply Theorem \ref{t:gouezel} to versions of $f$ and $h$ that are also zero on $\big(\Sigma^{\mathbb{N}} \times (\frac{1}{2},\frac{3}{4})\big) \backslash Y$.}

\vskip .2cm
For (g3), $h \in \mathcal H$ implies that $h \in \mathcal{H}_{\alpha}$ for some $\alpha \in (0,1)$ and thus $h \circ \pi \in \mathcal{C}_{1/2^\alpha}$ by Lemma \ref{lemma4.2d}. In particular this yields (g3) with $\delta = 2^{-\alpha} > \frac12$. For (g2), Lemma \ref{lemma4.1e}(iii) and \eqref{eq119y} give
\[ \mathbb P \times \mu (\varphi >n) = \int_{\{ \varphi >n\} \times \{0\}} \frac{d\nu}{dm} \, dm \le C_4 \cdot \mathbb P \times \lambda (\varphi >n).\]
Together with Proposition \ref{prop3.3b}(ii) this implies that
\[ \mathbb P \times \mu (\varphi >n) = O( n^{\gamma-1}).\]
Finally, (g1) follows from Lemma~\ref{lemma4.3p} by setting $C^* = C_6$ and noting that $\frac1{2^\alpha} > \frac{1}{2}$. Hence, we satisfy all the conditions of Theorem~\ref{t:gouezel} with $\delta = \frac1{2^\alpha}$ and $\zeta = 1-\gamma$. Note that
\[ K_{1-\gamma} (n) = \begin{cases}
n^{\gamma-1}, & \text{if } 1-\gamma >2,\\
\frac{\log n}{n^2}, & \text{if } 1-\gamma =2,\\
n^{2\gamma}, & \text{if } 1-\gamma \in (1,2).
\end{cases}\]
If $\gamma_1 < -1$, then $1-\gamma \in (\max\{-\gamma_2,2\},1-\gamma_1) \subseteq (2, \infty)$ and if $-1 \le  \gamma_1 <0$, then $1-\gamma \in (1-\frac{\gamma_2}{2}, 1-\gamma_1) \subseteq (1,2)$. We can thus conclude from Theorem \ref{t:gouezel} that
 \begin{align*}
 \Big|Cor_n(f,h) - \Big(\sum_{k>n}^{\infty} \mathbb{P} \times \mu(\varphi > k)\Big) \int f \, d \mathbb{P} \times \mu \int h \, d\mathbb{P} \times \mu \Big| = O(n^{\xi}),
 \end{align*}
where $\xi = \gamma-1$ if $\gamma_1 < -1$ and $\xi = 2\gamma$ if $-1 \leq \gamma_1 < 0$. As above it follows from Proposition \ref{prop3.3b}(iii) combined with Lemma \ref{lemma4.1e}(iii) and \eqref{eq119y} that $\mathbb{P} \times \mu(\varphi > n) = \Omega(n^{\gamma_2-1})$ and thus
\[ \sum_{k>n}^{\infty} \mathbb{P} \times \mu(\varphi > k) = \Omega(n^{\gamma_2}).\]
The result now follows from observing that $\gamma_2 > \xi$.
\end{proof}

We provide some examples of combinations of parameters for which the conditions of Theorem \ref{result2b} hold. As before set $\ell_{\min} = \min\{\ell_b\, :\,  b \in \Sigma_B\}$ and $p_B = \sum_{j \in \Sigma_B} p_j$ and set $\pi_B = \sum_{j \in \Sigma_B\, :\,  \ell_j = \ell_{\max}} p_j$. Examples that satisfy the conditions of Theorem \ref{result2b} include the following.
\begin{itemize}
\item[--] If $\Sigma_B$ consists of one element, then $\gamma_1 = \gamma_2$.
\item[--] If $p_B^{-1/3} < \ell_{\min} \leq \ell_{\max} < p_B^{-1/2}$, or equivalently $\ell_{\min} > \ell_{\max}^{2/3}$ and $p_B \in (\ell_{\min}^{-3},\ell_{\max}^{-2})$, then $\theta \leq p_B \cdot \ell_{\max} < \ell_{\max}^{-1}$, so $\theta < 1$ and $\gamma_1 < -1$, and
\[ \gamma_2 \geq 1 + \frac{\log p_B}{\log \ell_{\min}} > -2 \geq \gamma_1 -1.\]
\item[--] If $\pi_B > p_B^{4/3}$ (or equivalently $\pi_B^{-1/2} < p_B^{-2} \pi_B$) and $\ell_{\max} \in [\pi_B^{-1/2},p_B^{-2} \pi_B)$, then $\theta \leq p_B \cdot \ell_{\max} \leq p_B^{-1} \pi_B < 1$ and $\theta \geq \pi_B \cdot \ell_{\max} \geq \ell_{\max}^{-1}$, i.e.~$\gamma_1 \geq -1$, and
\[ \gamma_2 \geq 1 + \frac{\log \pi_B}{\log \ell_{\max}} > \frac{2 \log (p_B \cdot \ell_{\max})}{\log \ell_{\max}} \geq 2 \gamma_1.\]
\end{itemize}

\section{Further results and final remarks}\label{sec5}

We can obtain more information from the results from \cite{gouezel04}. First of all, using the last part of \cite[Theorem 6.3]{gouezel04} the upper bound in Theorem \ref{result2a} can be improved for a specific class of test functions.

\begin{theorem}\label{thrm5.1}
Assume that $\theta < 1$. Let $f \in L^{\infty}(\Sigma^{\mathbb{N}} \times [0,1], \mathbb{P} \times \mu)$ and $h \in \mathcal{H}$ be such that both $f$ and $h$ are identically zero on $\Sigma^{\mathbb{N}} \times \big([0,\frac{1}{2}] \cup [\frac{3}{4},1]\big)$ and $\int h \, d\mathbb{P} \times \mu = 0$. Let $\gamma \in (\gamma_1,0)$. Then
\begin{align*}
|Cor_n(f,h)| = O(n^{\gamma-1}).
\end{align*}
\end{theorem}

\begin{proof}
The statement follows by applying the last part of \cite[Theorem 6.3]{gouezel04}. For this, (g1), (g2) and (g3) need to be verified. This is done before in the proof of Theorem \ref{result2b}.
\end{proof}

In \cite[Theorem 6.13]{gouezel04} a Central Limit Theorem is derived for a specific class of functions in $\mathcal{H}$ with zero integral. This result immediately carries over to our setting and is given in the next theorem.

\begin{theorem}[cf.~Theorem 6.13 in \cite{gouezel04}] \label{thrm5.2}
Assume that $\theta < 1$. Let $h \in \mathcal{H}$ be identically zero on $\Sigma^{\mathbb{N}} \times \big([0,\frac{1}{2}] \cup [\frac{3}{4},1]\big)$ and with $\int h \, d\mathbb{P} \times \mu = 0$. Then the sequence $\frac{1}{\sqrt{n}} \sum_{k=0}^{n-1} h \circ F^k$ converges in distribution with respect to $\mathbb{P} \times \mu$ to a normally distributed random variable with zero mean and finite variance $\sigma^2$ given by
\begin{align*}
\sigma^2 = - \int h^2 \, d\mathbb{P} \times \mu + 2 \sum_{n=0}^{\infty} \int h \cdot h \circ F^n \, d \mathbb{P} \times \mu.
\end{align*}
Furthermore, we have $\sigma = 0$ if and only if there exists a measurable function $\psi$ on $\Sigma^{\mathbb{N}} \times [0,1]$ such that $h \circ F = \psi \circ F - \psi$. Such a function $\psi$ then satisfies $\sup_{z_1,z_2 \in Y} \frac{|\psi(z_1)-\psi(z_2)|}{(2^{-\alpha})^{s(z_1,z_2)}} < \infty$ and $\psi(F^j (z)) = \psi(z)$ for each $z \in Y$ and each $j = 0,1 \ldots, \varphi(z)-1$.
\end{theorem}

Using \cite[Theorem 4]{Y99} we can also derive a Central Limit Theorem, this time for a more general class of functions in $\mathcal{H}$ with zero integral but under the more restrictive assumption that $\theta < \ell_{\max}^{-1}$.

\begin{theorem}
Assume that $\theta < \ell_{\max}^{-1}$. Let $h \in \mathcal{H}$ be such that $\int h \, d\mathbb{P} \times \mu = 0$. Then the sequence $\frac{1}{\sqrt{n}} \sum_{k=0}^{n-1} h \circ F^k$ converges in distribution with respect to $\mathbb{P} \times \mu$ to a normally distributed random variable with zero mean and finite variance $\sigma^2$. Furthermore, we have $\sigma = 0$ if and only if there exists a measurable function $\psi$ on $\Delta$ such that $h \circ \pi \circ G = \psi \circ G - \psi$.
\end{theorem}

\begin{proof}
The result from \cite[Theorem 4]{Y99} gives a statement for $G$ on $\Delta$. We have already seen that $h \in \mathcal H$ implies $h \circ \pi \in \mathcal C_{1/2^{\alpha}}$ for some $\alpha \in (0,1)$. The assumption that $\theta < \ell_{\max}^{-1}$ implies $\gamma_1 < -1$. Take $\gamma \in (\gamma_1,-1)$. We saw in \eqref{eq140g} that $m(\hat{\varphi} > n) = O(n^{\gamma})$. It then follows from \cite[Theorem 4]{Y99} that $\frac{1}{\sqrt{n}} \sum_{k=0}^{n-1} h \circ \pi \circ G^k$ converges in distribution with respect to $\nu$ to a normally distributed random variable with zero mean and finite variance $\sigma^2$, with $\sigma > 0$ if and only if $h \circ \pi \circ G \neq \psi \circ G - \psi$ for any measurable function $\psi$ on $\Delta$. Since $\mathbb{P} \times \mu = \nu \circ \pi^{-1}$ and $F \circ \pi = \pi \circ G$, we get for any $u \in \mathbb{R}$ that
\[ \mathbb{P} \times \mu\Big(\frac{1}{\sqrt{n}} \sum_{k=0}^{n-1} h \circ F^k \leq u \Big) = \nu\Big( \frac{1}{\sqrt{n}} \sum_{k=0}^{n-1} (h \circ \pi) \circ G^k \leq u \Big).\]
The result now follows.
\end{proof}

Under an additional assumption on $r_{\min} = \min\{r_g: g \in \Sigma_G\}$ and $\ell_{\min} = \{\ell_b: b \in \Sigma_B\}$ we can weaken the assumption that the test functions in Theorem \ref{result2b}, Theorem \ref{thrm5.1} and Theorem \ref{thrm5.2} should be identically zero on $\Sigma^{\mathbb{N}} \times \big([0,\frac{1}{2}] \cup [\frac{3}{4},1]\big)$. Namely, if for an integer $l \geq 2$ we have
\begin{align*}
2^{-l} \cdot \min\Big\{ r_{\min} \cdot 2^{1/r_{\min}}, \ell_{\min}^{1+1/(\ell_{\min}-1)} \Big\} \geq 1,
\end{align*}
then it suffices to assume that these test functions are identically zero on $\Sigma^{\mathbb{N}} \times \big([0,\frac{1}{2}] \cup [1-\frac{1}{2^{l+1}},1]\big)$. Indeed, in this case Lemma \ref{lemma3.4a} and also Proposition \ref{prop3.3b} still carry over if we induce the random map on $(\frac{1}{2},1-\frac{1}{2^{l+1}})$ instead, and the result then follows by applying \cite[Theorem 6.3 and Theorem 6.13]{gouezel04} to this induced system in the same way as has been done in the proofs of Theorem \ref{result2b}, Theorem \ref{thrm5.1} and Theorem \ref{thrm5.2}. The step in Lemma \ref{lemma4.3p} where \eqref{eq119g} is applied will then be replaced by applying
\begin{align*}
\bigg| \frac{\frac{d\mu}{d\lambda}(x)}{\frac{d\mu}{d\lambda}(y)}-1\bigg| \leq C \cdot |x -y|, \qquad \forall x, y \in \Big[\frac{1}{2},1-\frac{1}{2^{l+1}}\Big], \qquad \text{ for some $C > 0$},
\end{align*}
which can be shown using that $\frac{d\mu}{d\lambda}$ is bounded away from zero and is locally Lipschitz on $[\frac{1}{2},1)$. As remarked in Remark~\ref{r:lipschitz}, the latter can be shown by following the same steps as in \cite[Section 3]{Zeegers21}. It in particular shows that $\frac{d\mu}{d\lambda}$ is bounded on $[\frac{1}{2},1-\frac{1}{2^{l+1}}]$, which replaces the step in the proof of Theorem \ref{result2b} where Lemma \ref{lemma4.1e}(iii) and \eqref{eq119y} are applied.

\vskip .2cm
We can extend the results in this article to the following more general classes of good and bad maps. Fix a $c \in (0,1)$, and let the class of good maps $\mathfrak G$ consist of maps $T_g: [0,1] \rightarrow [0,1]$ given by
\begin{align*}
T_g(x) =  \begin{cases}
1 - c^{-r_g} (c-x)^{r_g}, & \text{if}\quad x \in [0,c),\\
\frac{x-c}{1-c}, & \text{if}\quad x \in [c,1],
\end{cases}
\end{align*}
where $r_g \geq 1$, and the class of bad maps $\mathfrak B$ consist of maps $T_b: [0,1] \rightarrow [0,1]$ given by
\begin{align*}
T_b(x) =  \begin{cases}
c - c^{-\ell_b+1} (c-x)^{\ell_b}, & \text{if}\quad x \in [0,c),\\
\frac{x-c}{1-c}, & \text{if}\quad x \in [c,1],
\end{cases}
\end{align*}
where $\ell_b > 1$. For these collections of maps Lemma \ref{lemma2.6d} carries over replacing $\frac{1}{2}$ with $c$ and $\frac{1}{4}$ with $c^2$ under the additional assumptions that $c < r_g \cdot (1-c)^{1-r_g^{-1}}$ holds for all $g \in \Sigma_G$, and that $1-c > \ell_b^{1/(1-\ell_b)}$ holds for all $b \in \Sigma_B$. Furthermore, Lemma \ref{lemma3.4a} carries over under the additional assumption that
\begin{align*}
\frac{(1-c)^2}{c} \cdot \min\Big\{ r_{\min} \cdot (1-c)^{-1/r_{\min}}, \ell_{\min}^{1+1/(\ell_{\min}-1)} \Big\} \geq 1.
\end{align*}
By equipping $\Sigma^{\mathbb{N}}$ with the metric $d_{\Sigma^{\mathbb{N}}}(\omega,\omega') = (1-c)^{\min\{i \in \mathbb{N} \, : \,  \omega_i \neq \omega_i'\}}$, it can be shown that under these additional conditions all the results formulated in Sections \ref{sec1} and \ref{sec5} carry over and are proven in the same way.

\vskip .2cm
Finally, polynomial decay of correlations is expected to hold for a more general class of good and bad maps for which random compositions show critical intermittency, but the proofs may become more cumbersome. Our assumption that all maps are identical on the interval $[\frac12, 1]$ made it easier to find a suitable inducing domain, but does not seem necessary. Furthermore, the linearity of this right branch and the explicit forms of the left branches of the good and bad maps made the series in \eqref{eq77t} and \eqref{eq80v} telescopic. A first step to generalise our results to a more general class might be to require this explicit form of the left branch only close to $c$, though any generalisations will inevitably make the calculations more complicated.

\section*{Acknowledgments}
We would like to thank Marks Ruziboev for valuable discussions.


\bibliographystyle{plain} 
\bibliography{Bibliography}

\end{document}